\documentclass[nosumlimits,twoside]{amsart}
\usepackage{amsfonts, amsmath, amssymb}
\usepackage[bookmarksnumbered,plainpages,driverfallback=dvipdfm,backref=page]{hyperref}
\usepackage{graphicx}
\usepackage{float}
\usepackage{srcltx}
\usepackage[all]{xy}
\usepackage{version}
\usepackage[T1]{fontenc}
\usepackage{soul, color}

\newtheorem{theorem}{\sc Theorem}[section]
\newtheorem{proposition}[theorem]{\sc Proposition}

\newtheorem{lemma}[theorem]{\sc Lemma}
\newtheorem{corollary}[theorem]{\sc Corollary}
\theoremstyle{definition}
\newtheorem{definition}[theorem]{\sc Definition}
\newtheorem{definitions}[theorem]{\sc Definitions}
\newtheorem{example}[theorem]{\sc Example}

\theoremstyle{remark}
\newtheorem{remark}[theorem]{\sc Remark}

\newtheorem{claim}[theorem]{}

\newcommand\id{\mathrm{Id}}
\newcommand\M{\mathcal{M}}

\newcommand\Alg{\mathrm{Alg}}
\newcommand\Bialg{\mathrm{Bialg}}

\renewcommand\vec{\mathrm{Vec}_{\Bbbk }}

\allowdisplaybreaks
\IfFileExists{tcilatex.tex}{\input{tcilatex}}{}
\newenvironment{invisible}{{\noindent\sc \underline{\color{blue}Invisible (To be hidden)}:\quad}\color{red}}{\medskip}

\excludeversion{invisible}

\begin{document}
\title{Heavily Separable Functors}
\author{Alessandro Ardizzoni}
\address{%
\parbox[b]{\linewidth}{University of Turin, Department of Mathematics ``G. Peano'', via
Carlo Alberto 10, I-10123 Torino, Italy}}
\email{alessandro.ardizzoni@unito.it}
\urladdr{sites.google.com/site/aleardizzonihome}
\author{Claudia Menini}
\address{University of Ferrara, Department of Mathematics, Via Machiavelli
35, Ferrara, I-44121, Italy}
\email{men@unife.it}
\urladdr{sites.google.com/a/unife.it/claudia-menini}
\thanks{This paper was written while both authors were members of the
"National Group for Algebraic and Geometric Structures and their
Applications" (GNSAGA-INdAM)}

\begin{abstract}
Prompted by an example related to the tensor algebra, we introduce and
investigate a stronger version of the notion of separable functor that we
call heavily separable. We test this notion on several functors
traditionally connected to the study of separability.
\end{abstract}

\maketitle

\section*{Introduction}

Given a field $\Bbbk$, the functor $\mathbf{P}:\Bialg_\Bbbk\to \vec$, assigning to a $\Bbbk$%
-bialgebra $B$ the $\Bbbk$-vector space of its primitive elements, admits a
left adjoint $\mathbf{T}$, assigning to a vector space $V$ the tensor
algebra $\mathbf{T}V$ endowed with its canonical bialgebra structure such
that the elements in $V$ becomes primitive. By investigating the properties
of the adjunction $(\mathbf{T},\mathbf{P})$, together with its unit $%
\boldsymbol{\eta}$ and counit $\boldsymbol{\epsilon}$, we discovered that
there is a natural retraction $\boldsymbol{\gamma}:\mathbf{T}\mathbf{P}\to
\mathrm{Id}$ of $\boldsymbol{\eta}$, i.e. $\boldsymbol{\gamma}\circ
\boldsymbol{\eta}=\mathrm{Id}$, fulfilling the condition $\boldsymbol{\gamma}%
\boldsymbol{\gamma}=\boldsymbol{\gamma}\circ \mathbf{P}\boldsymbol{\epsilon}%
\mathbf{T}$. The existence of a natural retraction of the unit of an
adjunction is, by Rafael Theorem, equivalent to the fact that the left
adjoint is a separable functor. It is then natural to wonder if the above
extra condition on the retraction $\boldsymbol{\gamma}$ corresponds to a
stronger notion of separability. In the present paper, we show that an
affirmative answer to this question is given by what we call a heavily
separable (h-separable for short) functor and we investigate this notion in case of functors
usually connected to the study of separability.

Explicitly, in Section \ref{sec:1} we introduce the concept of h-separable functor and we recover classical results in the h-separable case such as their behaviour with respect to composition (Lemma \ref{lem:compheav}). In Section \ref{sec:2}, we obtain a Rafael type Theorem \ref{thm:Rafael}. As a consequence we characterize the h-separability of a left (respectively right) adjoint functor  either with respect to the forgetful functor from the Eilenberg-Moore category of the associated monad (resp. comonad) in Proposition \ref{pro:heavsepU} or by the existence of an augmentation (resp. grouplike morphism) of the associated monad (resp. comonad) in Corollary \ref{coro:hsepaug}. In Theorem \ref{thm:indu}, we prove that the induced functor attached to an $A$-coring is h-separable if and only if this coring has an invariant grouplike element.

Section \ref{sec:3} is devoted to the investigation of the h-separability of the induction functor $\varphi^*$ and of the restriction of scalars functors $\varphi_*$ attached to a ring homomorphism $\varphi:R\to S.$
In Proposition \ref{pro:inducfunhsep}, we prove that $\varphi^*$ is h-separable if and only if there is a ring homomorphism $E:S\to R$ such that $E\circ \varphi=\id$.
Characterizing whether $\varphi_*$ is h-separable (in this case we say that $S/R$ is h-separable) is more laborious.  In Proposition \ref{pro:S/R}, we prove that $S/R$ is h-separable if and only if it is endowed with what we call a $h$-separability idempotent, a stronger version of a separability idempotent. In Lemma \ref{lem:ringepim} we show that the ring epimorphisms (by this we mean epimorphisms in the category of rings) provide particular examples of h-separability. In Lemma \ref{lem:matrix} we show that the ring of  matrices is never h-separable over the base ring except in trivial cases. In the rest of the present section we investigate the particular case when $S$ is an $R$-algebra i.e. $\mathrm{Im}(\varphi)\subseteq Z(S).$ In Theorem \ref{thm:alg} we discover that, in this case, $S/R$ is h-separable if and only if $\varphi$ is a ring epimorphism. Moreover $S$ becomes commutative. As a consequence, in Proposition \ref{Pro:hsepoverfield} we show that a h-separable algebra over a field $\Bbbk$ is necessarily trivial.

Finally in Section \ref{sec:4} we provide a more general version of our starting example $(\mathbf{T%
},\mathbf{P})$ involving monoidal categories and bialgebras therein.

\section{Heavily separable functors}\label{sec:1}

In this section we collect general facts about heavily separable functors.

\begin{definition}
\label{def:heavsep}For every functor $F:\mathcal{B}\rightarrow \mathcal{A}$
we set%
\begin{equation*}
F_{X,Y}:\mathrm{Hom}_{\mathcal{B}}\left( X,Y\right) \rightarrow \mathrm{Hom}%
_{\mathcal{A}}\left( FX,FY\right) :f\mapsto Ff
\end{equation*}%
Recall that $F$ is called \textbf{separable} if there is a natural transformation
\begin{equation*}
P_{-,-}:=P_{-,-}^{F}:\mathrm{Hom}_{\mathcal{A}}\left( F-,F-\right)
\rightarrow \mathrm{Hom}_{\mathcal{B}}\left( -,-\right)
\end{equation*}%
such that $P_{X,Y}\circ F_{X,Y}=\mathrm{Id}$ for every $X,Y$ objects in $\mathcal{B}$.

We say that $F$ is \textbf{heavily separable } (\textbf{h-separable} for
short) if it is separable and the $P_{X,Y}$'s make commutative the following
diagram for every $X,Y,Z\in \mathcal{B}.$%
\begin{equation*}
\xymatrixcolsep{1.5cm}\xymatrixrowsep{0.7cm} \xymatrix{\mathrm{Hom}_{%
\mathcal{A}}\left( FX,FY\right) \times \mathrm{Hom}_{\mathcal{A}}\left(
FY,FZ\right)\ar[rr]^{P_{X,Y}\times
P_{Y,Z}}\ar[d]_{\circ}&&\mathrm{Hom}_{\mathcal{B}}\left( X,Y\right) \times
\mathrm{Hom}_{\mathcal{B}}\left( Y,Z\right)\ar[d]^{\circ}\\
\mathrm{Hom}_{\mathcal{A}}\left(
FX,FZ\right)\ar[rr]^{P_{X,Z}}&&\mathrm{Hom}_{\mathcal{B}}\left( X,Z\right) }
\end{equation*}

\begin{invisible}
\begin{equation*}
\begin{array}{ccc}
\mathrm{Hom}_{\mathcal{A}}\left( FX,FY\right) \times \mathrm{Hom}_{\mathcal{A%
}}\left( FY,FZ\right) & \overset{P_{X,Y}\times P_{Y,Z}}{\longrightarrow } &
\mathrm{Hom}_{\mathcal{B}}\left( X,Y\right) \times \mathrm{Hom}_{\mathcal{B}%
}\left( Y,Z\right) \\
\circ \downarrow &  & \downarrow \circ \\
\mathrm{Hom}_{\mathcal{A}}\left( FX,FZ\right) & \overset{P_{X,Z}}{%
\longrightarrow } & \mathrm{Hom}_{\mathcal{B}}\left( X,Z\right)%
\end{array}%
\end{equation*}
\end{invisible}

where the vertical arrows are the obvious compositions. On elements the
above diagram means that $P_{X,Z}(f\circ g)=P_{Y,Z}(f)\circ P_{X,Y}(g).$
\end{definition}

\begin{remark}
We were tempted to use the word "strongly" at first, instead of "heavily",
but a notion of "strongly separable functor" already appeared in the
literature in connection with graded rings in \cite[Definition 3.1]{CGN}.
\end{remark}

\begin{remark}
The Maschke's Theorem for separable functors asserts that for a separable
functor $F:\mathcal{B}\rightarrow \mathcal{A}$ a morphism $f:X\rightarrow Y$
splits (resp. cosplits) if and only if $F\left( f\right) $ does. Explicitly,
if $F\left( f\right) \circ g=\mathrm{Id}$ (resp. $g\circ F\left( f\right) =%
\mathrm{Id}$) for some morphism $g$ then $f\circ P_{Y,X}\left( g\right) =%
\mathrm{Id}$ (resp. $P_{Y,X}\left( g\right) \circ f=\mathrm{Id}$). If $%
F\left( f\right) \circ g=\mathrm{Id}$ and $F\left( f^{\prime }\right) \circ
g^{\prime }=\mathrm{Id}$ for $f:X\rightarrow Y,f^{\prime }:Y\rightarrow Z,$
then $f\circ P_{Y,X}\left( g\right) =\mathrm{Id}$ and $f^{\prime }\circ
P_{Z,Y}\left( g^{\prime }\right) =\mathrm{Id}$ so that $f^{\prime }\circ
f\circ P_{Y,X}\left( g\right) \circ P_{Z,Y}\left( g^{\prime }\right) =%
\mathrm{Id}$ so that $P_{Y,X}\left( g\right) \circ P_{Z,Y}\left( g^{\prime
}\right) $ is a section of $f^{\prime }\circ f.$ Since $F\left( f^{\prime
}\circ f\right) \circ g\circ g^{\prime }=\mathrm{Id},$ we also have $%
f^{\prime }\circ f\circ P_{Z,X}\left( g\circ g^{\prime }\right) =\mathrm{Id}$
so that $P_{Z,X}\left( g\circ g^{\prime }\right) $ is another section of $%
f^{\prime }\circ f.$ In general these two section may differ but not in case
$F$ is h-separable. Thus in some sense we get a sort of functoriality of the
splitting. A similar remark holds for cosplittings. We thank J. Vercruysse
for this observation.
\end{remark}

\begin{lemma}
\label{lem:compheav}Let $F:\mathcal{C}\rightarrow \mathcal{B}$ and $G:%
\mathcal{B}\rightarrow \mathcal{A}$ be functors.

\begin{enumerate}
\item[i)] If $F$ and $G$ are h-separable so is $GF.$

\item[ii)] If $GF$ is h-separable so is $F.$

\item[iii)] If $G$ is h-separable, then $F$ is h-sparable if and only if so
is $GF.$
\end{enumerate}
\end{lemma}

\begin{proof}
i). By \cite[Lemma 1]{NVV} we know that $GF$ is separable with respect to $%
P_{X,Y}^{GF}:=P_{X,Y}^{F}\circ P_{FX,FY}^{G}.$ As a consequence, since $F$
and $G$ are h-separable, the following diagram {\footnotesize
\begin{equation*}
\xymatrixcolsep{1.5cm}\xymatrixrowsep{0.7cm}\xymatrix{\mathrm{Hom}_{%
\mathcal{A}}\left( GFX,GFY\right) \times \mathrm{Hom}_{\mathcal{A}}\left(
GFY,GFZ\right)\ar[r]^{P^G_{FX,FY}\times
P^G_{FY,FZ}}\ar[d]_{\circ}&\mathrm{Hom}_{\mathcal{B}}\left( FX,FY\right)
\times \mathrm{Hom}_{\mathcal{B}}\left(
FY,FZ\right)\ar[d]^{\circ}\ar[r]^-{P^F_{X,Y}\times
P^F_{Y,Z}}&\mathrm{Hom}_{\mathcal{C}}\left( X,Y\right) \times
\mathrm{Hom}_{\mathcal{C}}\left( Y,Z\right)\ar[d]^{\circ}\\
\mathrm{Hom}_{\mathcal{A}}\left(
GFX,GFZ\right)\ar[r]^{P^G_{FX,FZ}}&\mathrm{Hom}_{\mathcal{B}}\left(
FX,FZ\right)\ar[r]^{P^F_{X,Z}}&\mathrm{Hom}_{\mathcal{C}}\left( X,Z\right) }
\end{equation*}%
}

\begin{invisible}
{\tiny
\begin{equation*}
\begin{array}{ccccc}
\mathrm{Hom}_{\mathcal{A}}\left( GFX,GFY\right) \times \mathrm{Hom}_{%
\mathcal{A}}\left( GFY,GFZ\right) & \overset{P_{FX,FY}^{G}\times
P_{FY,FZ}^{G}}{\longrightarrow } & \mathrm{Hom}_{\mathcal{B}}\left(
FX,FY\right) \times \mathrm{Hom}_{\mathcal{B}}\left( FY,FZ\right) & \overset{%
P_{X,Y}^{F}\times P_{Y,Z}^{F}}{\longrightarrow } & \mathrm{Hom}_{\mathcal{C}%
}\left( X,Y\right) \times \mathrm{Hom}_{\mathcal{C}}\left( Y,Z\right) \\
\circ \downarrow &  & \circ \downarrow &  & \downarrow \circ \\
\mathrm{Hom}_{\mathcal{A}}\left( GFX,GFZ\right) & \overset{P_{FX,FZ}^{G}}{%
\longrightarrow } & \mathrm{Hom}_{\mathcal{B}}\left( FX,FZ\right) & \overset{%
P_{X,Z}^{F}}{\longrightarrow } & \mathrm{Hom}_{\mathcal{C}}\left( X,Z\right)%
\end{array}%
\end{equation*}%
}
\end{invisible}

commutes so that $GF$ is h-separable.

ii). By \cite[Lemma 1]{NVV} we know that $F_{X,Y}$ cosplits naturally
through $P_{X,Y}^{F}:=P_{X,Y}^{GF}\circ G_{FX,FY}.$ On the other hand, since
$G$ is a functor and $GF$ is h-separable the following diagram commutes
{\footnotesize
\begin{equation*}
\xymatrixcolsep{1.5cm}\xymatrixrowsep{0.7cm}\xymatrix{\mathrm{Hom}_{%
\mathcal{B}}\left( FX,FY\right) \times \mathrm{Hom}_{\mathcal{B}}\left(
FY,FZ\right)\ar[r]^{G_{FX,FY}\times
G_{FY,FZ}}\ar[d]_{\circ}&\mathrm{Hom}_{\mathcal{A}}\left( GFX,GFY\right)
\times \mathrm{Hom}_{\mathcal{A}}\left(
GFY,GFZ\right)\ar[d]^{\circ}\ar[r]^-{P^{GF}_{X,Y}\times
P^{GF}_{Y,Z}}&\mathrm{Hom}_{\mathcal{C}}\left( X,Y\right) \times
\mathrm{Hom}_{\mathcal{C}}\left( Y,Z\right)\ar[d]^{\circ}\\
\mathrm{Hom}_{\mathcal{B}}\left(
FX,FZ\right)\ar[r]^{G_{FX,FZ}}&\mathrm{Hom}_{\mathcal{A}}\left(
GFX,GFZ\right)\ar[r]^{P^{GF}_{X,Z}}&\mathrm{Hom}_{\mathcal{C}}\left(
X,Z\right) }
\end{equation*}%
}

\begin{invisible}
{\tiny
\begin{equation*}
\begin{array}{ccccc}
\mathrm{Hom}_{\mathcal{B}}\left( FX,FY\right) \times \mathrm{Hom}_{\mathcal{B%
}}\left( FY,FZ\right) & \overset{G_{FX,FY}\times G_{FY,FZ}}{\longrightarrow }
& \mathrm{Hom}_{\mathcal{A}}\left( GFX,GFY\right) \times \mathrm{Hom}_{%
\mathcal{A}}\left( GFY,GFZ\right) & \overset{P_{X,Y}^{GF}\times P_{Y,Z}^{GF}}%
{\longrightarrow } & \mathrm{Hom}_{\mathcal{C}}\left( X,Y\right) \times
\mathrm{Hom}_{\mathcal{C}}\left( Y,Z\right) \\
\circ \downarrow &  & \circ \downarrow &  & \downarrow \circ \\
\mathrm{Hom}_{\mathcal{B}}\left( FX,FZ\right) & \overset{G_{FX,FZ}}{%
\longrightarrow } & \mathrm{Hom}_{\mathcal{A}}\left( GFX,GFZ\right) &
\overset{P_{X,Z}^{GF}}{\longrightarrow } & \mathrm{Hom}_{\mathcal{C}}\left(
X,Z\right)%
\end{array}%
\end{equation*}%
}
\end{invisible}

so that $F$ is h-separable.

iii) It follows trivially from i) and ii).
\end{proof}

\begin{remark}
The present remark was pointed out by J. Vercruysse. If the functor $F:%
\mathcal{B}\rightarrow \mathcal{A}$ is a split monomorphism, meaning that there is a
functor $G:\mathcal{A}\rightarrow \mathcal{B}$ such that $G\circ F=\mathrm{Id%
},$ then $F$ is h-separable. This follows by setting $P_{X,Y}:=G_{X,Y}$ as
in Definition \ref{def:heavsep}. It can also be proved by means of Lemma \ref%
{lem:compheav},ii).
\end{remark}

\begin{lemma}
\label{lem:ffhsep}A full and faithful functor is h-separable.
\end{lemma}

\begin{proof}
If $F:\mathcal{B}\rightarrow \mathcal{A}$ is full and faithful we have that
the canonical map $F_{X,Y}:\mathrm{Hom}_{\mathcal{B}}\left( X,Y\right)
\rightarrow \mathrm{Hom}_{\mathcal{A}}\left( FX,FY\right) $ is invertible so
that we can take $P_{X,Y}:=F_{X,Y}^{-1}.$ Since $F$ is a functor, the
following diagram commutes
\begin{equation*}
\xymatrixcolsep{1.5cm}\xymatrixrowsep{0.7cm}\xymatrix{
\mathrm{Hom}_{\mathcal{B}}\left( X,Y\right) \times
\mathrm{Hom}_{\mathcal{B}}\left( Y,Z\right)\ar[rr]^{F_{X,Y}\times
F_{Y,Z}}\ar[d]_{\circ}&&\mathrm{Hom}_{\mathcal{A}}\left( FX,FY\right) \times
\mathrm{Hom}_{\mathcal{A}}\left(
FY,FZ\right)\ar[d]^{\circ}\\\mathrm{Hom}_{\mathcal{B}}\left( X,Z\right)
\ar[rr]^{F_{X,Z}}&&\mathrm{Hom}_{\mathcal{A}}\left( FX,FZ\right) }
\end{equation*}

\begin{invisible}
\begin{equation*}
\begin{array}{ccc}
\mathrm{Hom}_{\mathcal{B}}\left( X,Y\right) \times \mathrm{Hom}_{\mathcal{B}%
}\left( Y,Z\right)  & \overset{F_{X,Y}\times F_{Y,Z}}{\longrightarrow } &
\mathrm{Hom}_{\mathcal{A}}\left( FX,FY\right) \times \mathrm{Hom}_{\mathcal{A%
}}\left( FY,FZ\right)  \\
\downarrow \circ  &  & \circ \downarrow  \\
\mathrm{Hom}_{\mathcal{B}}\left( X,Z\right)  & \overset{F_{X,Z}}{%
\longrightarrow } & \mathrm{Hom}_{\mathcal{A}}\left( FX,FZ\right)
\end{array}%
\end{equation*}
\end{invisible}

Reversing the horizontal arrows we obtain that $F$ h-separable.
\end{proof}

\section{Heavily separable adjoint functors}\label{sec:2}

In this section we investigate h-separable functors which are adjoint
functors.

\begin{theorem}[Rafael type Theorem]
\label{thm:Rafael}Let $\left( L,R,\eta ,\epsilon \right) $ be an adjunction
where $L:\mathcal{B}\rightarrow \mathcal{A}.$

\begin{enumerate}
\item[(i)] $L$ is h-separable if and only if there is a natural
transformation $\gamma :RL\rightarrow \mathrm{Id}_{\mathcal{B}}$ such that $%
\gamma \circ \eta =\mathrm{Id}$ and%
\begin{equation}
\gamma \gamma =\gamma \circ R\epsilon L.  \label{form:gamma0}
\end{equation}

\item[(ii)] $R$ is h-separable if and only if there is a natural
transformation $\delta :\mathrm{Id}_{\mathcal{A}}\rightarrow LR$ such that $%
\epsilon \circ \delta =\mathrm{Id}$ and%
\begin{equation}
\delta \delta =L\eta R\circ \delta .  \label{form:delta0}
\end{equation}
\end{enumerate}
\end{theorem}

\begin{proof}
The second part of the statement follows from the first by duality so we
only have to establish (i).

First recall that, by Rafael Theorem \cite[Theorem 1.2]{Rafael}, the
functor $L$ is separable if and only if there is a natural transformation $%
\gamma :RL\rightarrow \mathrm{Id}_{\mathcal{B}}$ such that $\gamma \circ
\eta =\mathrm{Id.}$ Moreover, by construction
\begin{equation}
\gamma X=P_{RLX,X}\left( \epsilon LX\right)  \label{form:gammaP}
\end{equation}%
so that, by naturality of $P_{-,-}$ one has
\begin{equation*}
\gamma Y\circ Rg\circ \eta X=P_{RLY,Y}\left( \epsilon LY\right) \circ
Rg\circ \eta X=P_{X,Y}\left( \epsilon LY\circ LRg\circ L\eta X\right)
=P_{X,Y}\left( g\circ \epsilon LX\circ L\eta X\right) =P_{X,Y}(g).
\end{equation*}%
Assume that (\ref{form:gamma0}) holds. Then, for all $f\in \mathrm{Hom}_{%
\mathcal{A}}\left( LX,LY\right) $ and $g\in \mathrm{Hom}_{\mathcal{A}}\left(
LY,LZ\right) ,$ we have
\begin{eqnarray*}
P_{Y,Z}\left( g\right) \circ P_{X,Y}\left( f\right) &=&\gamma Z\circ Rg\circ
\eta Y\circ \gamma Y\circ Rf\circ \eta X \\
&\overset{\text{nat.}\gamma }{=}&\gamma Z\circ \gamma RL Z\circ RLRg\circ
RL\eta Y\circ Rf\circ \eta X \\
&\overset{\left( \ref{form:gamma0}\right) }{=}&\gamma Z\circ R\epsilon
LZ\circ RLRg\circ RL\eta Y\circ Rf\circ \eta X \\
&=&\gamma Z\circ Rg\circ R\epsilon LY\circ RL\eta Y\circ Rf\circ \eta X \\
&=&\gamma Z\circ Rg\circ Rf\circ \eta X=P_{X,Z}\left( g\circ f\right)
\end{eqnarray*}%
so that $P_{Y,Z}\left( g\right) \circ P_{X,Y}\left( f\right) =P_{X,Z}\left(
g\circ f\right) $ and hence $L$ is h-separable. Conversely, if the latter
condition holds for every $X,Y,Z$ and $f,g$ as above, we have
\begin{eqnarray*}
&&\gamma \gamma X=\gamma X\circ \gamma RLX\overset{(\ref{form:gammaP})}{=}%
P_{RLX,X}\left( \epsilon LX\right) \circ P_{RLRLX,RLX}\left( \epsilon
LRLX\right) \\
&=&P_{RLRLX,X}\left( \epsilon LX\circ \epsilon LRLX\right) =\gamma X\circ
R\left( \epsilon LX\circ \epsilon LRLX\right) \circ \eta RLRLX \\
&=&\gamma X\circ R\epsilon LX\circ R\epsilon LRLX\circ \eta RLRLX=\gamma
X\circ R\epsilon LX
\end{eqnarray*}%
so that (\ref{form:gamma0}) holds.
\end{proof}

\begin{remark}
Let $\gamma $ be as in Theorem \ref{thm:Rafael}. Then%
\begin{equation*}
\xymatrixcolsep{0.7cm}\xymatrixrowsep{0.7cm}
\xymatrix{RLRL\ar@<.5ex>[rr]^{RL\gamma}\ar@<-.5ex>[rr]_{R\epsilon L}&&RL\ar[r]^{\gamma}&\id}
\end{equation*}
\begin{invisible}
\begin{equation*}
\begin{array}{ccccc}
RLRL & \overset{RL\gamma }{\underset{R\epsilon L}{\rightrightarrows }} & RL
& \overset{\gamma }{\longrightarrow } & \mathrm{Id}%
\end{array}%
\end{equation*}%
\end{invisible}
is a split coequalizer (in the sense of \cite[Definition 4.4.2 ]{Borceux2})
taking $u=R\epsilon L,v=RL\gamma ,r=\eta RL,q=\gamma $ and $s=\eta .$ In
particular the coequalizer above is absolute, i.e. preserved by every
functor defined on $\mathcal{B}$. One can see this as a consequence of the
fact that, for every $B\in \mathcal{B},$ the pair $\left( B,\gamma B\right)$ belongs to $_{RL}\mathcal{B}$, i.e. the Eilenberg-Moore category of the monad $\left(RL,R\epsilon L,\eta \right) $, and hence \cite[Lemma 4.4.3 ]{Borceux2} applies to this pair. A similar remark holds for $\delta $ as in Theorem \ref{thm:Rafael} in connection with the Eilenberg-Moore category $\mathcal{B}^{LR}$ of the comonad $\left( LR,L\eta
R,\epsilon \right) $.
\end{remark}

\begin{proposition}
\label{pro:heavsepU} Let $\left( L,R\right) $ be an adjunction.

\begin{itemize}
\item[$\left( 1\right) $] The functor $L$ is h-separable if and only if the
forgetful functor $U:{_{RL}\mathcal{B}}\rightarrow \mathcal{B}$ is a split epimorphism i.e. there is
a functor $\Gamma :\mathcal{B}\rightarrow {_{RL}\mathcal{B}}$ such that $%
U\circ \Gamma =\mathrm{Id}_{\mathcal{B}}.$

\item[$\left( 2\right) $] The functor $R$ is h-separable if and only if the
functor forgetful $U:\mathcal{B}^{LR}\rightarrow \mathcal{B}$ is a split epimorphism i.e. there is a functor
$\Gamma :\mathcal{B}\rightarrow \mathcal{B}^{LR}$ such that $U\circ \Gamma =%
\mathrm{Id}_{\mathcal{B}}.$
\end{itemize}
\end{proposition}

\begin{proof}
We just prove $\left( 1\right) ,$ the proof of $\left( 2\right) $ being
similar. By Theorem \ref{thm:Rafael}, $L$ is h-separable if and only if
there is a natural transformation $\gamma :RL\rightarrow \mathrm{Id}_{%
\mathcal{B}}$ such that $\gamma \circ \eta =\mathrm{Id}$ and (\ref%
{form:gamma0}) holds. For every $B\in \mathcal{B},$ define $\Gamma B:=\left(
B,\gamma B\right) .$ Then $\Gamma B\in {_{RL}\mathcal{B}}$ by the properties
of $\gamma .$ Moreover any morphism $f:B\rightarrow C$ fulfills $f\circ
\gamma B=\gamma C\circ RLf$ by naturality of $\gamma .$ This means that $f$
induces a morphism $\Gamma f:\Gamma B\rightarrow \Gamma C$ such that $%
U\Gamma f=f.$ We have so defined a functor $\Gamma :\mathcal{B}\rightarrow
\mathcal{B}_{RL}$ such that $U\circ \Gamma =\mathrm{Id}_{\mathcal{B}}.$

Conversely, let $\Gamma :\mathcal{B}\rightarrow {_{RL}\mathcal{B}}$ be a
functor such that $U\circ \Gamma =\mathrm{Id}_{\mathcal{B}}.$ Then, for
every $B\in \mathcal{B},$ we have that $\Gamma B=\left( B,\gamma B\right) $
for some morphism $\gamma B:RLB\rightarrow B.$ Since $\Gamma B\in \mathcal{B}%
_{RL}$ we must have that $\gamma B\circ \eta B=B$ and $\gamma B\circ
RL\gamma B=\gamma B\circ R\epsilon LB.$ Given a morphism $f:B\rightarrow C,$
we have that $\Gamma f:\Gamma B\rightarrow \Gamma C$ is a morphism in ${_{RL}%
\mathcal{B}},$ which means that $f\circ \gamma B=\gamma C\circ RLf$ i.e. $%
\gamma :=\left( \gamma B\right) _{B\in \mathcal{B}}$ is a natural
transformation. By the foregoing $\gamma \circ \eta =\mathrm{Id}$ and (\ref%
{form:gamma0}) holds.
\end{proof}

\begin{corollary}
\label{coro:heavsepU}Let $\left( L,R\right) $ be an adjunction.

\begin{itemize}
\item[$\left( 1\right) $] Assume that $R$ is strictly monadic (i.e. the
comparison functor is an isomorphism of categories). Then the functor $L$ is
h-separable if and only if $R$ is a split epimorphism.

\item[$\left( 2\right) $] Assume that $L$ is strictly comonadic (i.e. the
cocomparison functor is an isomorphism of categories). Then the functor $R$
is h-separable if and only if $L$ is a split epimorphism.
\end{itemize}
\end{corollary}

\begin{proof}
We just prove $\left( 1\right) ,$ the proof of $\left( 2\right) $ being
similar. Since the comparison functor $K:\mathcal{A}\rightarrow \mathcal{B}%
_{RL}$ is an isomorphism of categories and $U\circ K=R$, we get that $R$ is a
split epimorphism if and only if $U$ is a split epimorphism.
By Proposition \ref{pro:heavsepU}, this is equivalent to the h-separability
of $L.$
\end{proof}

\begin{example}
In Remark \ref{rem:Tnotheavy} we will obtain that the tensor algebra functor $T:\mathrm{Vec}%
_{\Bbbk }\rightarrow \mathrm{Alg}_{\Bbbk }$ is
separable but not h-separable.
\end{example}

\begin{definition}
Following \cite[Section 4]{LMW} we say that a \textbf{grouplike morphism}
for a comonad $\left( C,\Delta :C\rightarrow CC,\epsilon :C\rightarrow
\mathrm{Id}\right) $ is a natural transformation $\delta :\mathrm{Id}%
\rightarrow C$ such that $\epsilon \circ \delta =\mathrm{Id}$ and $\delta
\delta =\Delta \circ \delta . $ Dually an \textbf{augmentation} for a monad $%
\left( M,m:MM\rightarrow M,\eta :\mathrm{Id}\rightarrow M\right) $ is a
natural transformation $\gamma :M\rightarrow \mathrm{Id}$ such that $\gamma
\circ \eta =\mathrm{Id}$ and $\gamma \gamma =\gamma \circ m.$
\end{definition}

An immediate consequence of the previous definition and Theorem \ref%
{thm:Rafael} is the following result.

\begin{corollary}
\label{coro:hsepaug}Let $\left( L,R,\eta ,\epsilon \right) $ be an
adjunction.

\begin{enumerate}
\item $L$ is h-separable if and only if the monad $\left( RL,R\epsilon
L,\eta \right) $ has an augmentation.

\item $R$ is h-separable if and only if the comonad $\left( LR,L\eta
R,\epsilon \right) $ has a grouplike morphism.
\end{enumerate}
\end{corollary}

Consider an $A$-coring $\mathcal{C}$ and its set of invariant elements $%
\mathcal{C}^{A}=\left\{ c\in \mathcal{C}\mid ac=ca,\text{ for every }a\in
A\right\} $. In \cite[Theorem 3.3]{Br}, Brzezi\'{n}ski proved that the
induction functor $R:=\left( -\right) \otimes _{A}\mathcal{C}:$ \textrm{Mod}-%
$A\rightarrow \mathcal{M}^{\mathcal{C}}$ is separable if and only if there
is an invariant element $e\in \mathcal{C}^{A}$ such that $\varepsilon _{%
\mathcal{C}}\left( e\right) =1.$ Next result provides a similar
characterization for the h-separable case.

\begin{theorem}
\label{thm:indu}Given an $A$-coring $\mathcal{C}$, the induction functor $%
R:=\left( -\right) \otimes _{A}\mathcal{C}:\mathrm{Mod}$-$A\rightarrow
\mathcal{M}^{\mathcal{C}}$ is h-separable if and only if $\mathcal{C}$ has
an invariant grouplike element.
\end{theorem}

\begin{proof}
By \cite[Lemma 3.1]{Br}, the functor $R$ is the right adjoint of the
forgetful functor $L:\mathcal{M}^{\mathcal{C}}\rightarrow \mathcal{M}_{A}$.
Thus by Corollary \ref{coro:hsepaug}, $R$ is h-separable if and only if the
comonad $\left( LR,L\eta R,\epsilon \right) $ has a grouplike morphism. A
grouplike morphism for this particular comonad is equivalent to an invariant
grouplike element for the coring $\mathcal{C}$ i.e. an element $e\in
\mathcal{C}^{A}$ such that $\varepsilon _{\mathcal{C}}\left( e\right) =1$
and $\Delta_{\mathcal{C}} \left( e\right) =e\otimes _{A}e.$

\begin{invisible}
The comonad $\left( LR,L\eta R,\epsilon \right) $ has a grouplike morphism
if and only if there is a natural transformation $\nu :\mathrm{Id}\rightarrow LR$ such $%
\epsilon \circ \nu =\mathrm{Id}$ and $\nu \nu =L\eta R\circ \nu .$ Let $M\in
\mathcal{M}_{A}$ and $m\in M.$ Consider $f_{m}:A\rightarrow M:a\mapsto ma.$
By naturality of $\nu $ we have that $\nu M\circ f_{m}=LRf_{m}\circ \nu A$
i.e.%
\begin{equation*}
\left( \nu M\right) \left( ma\right) =\left( f_{m}\otimes _{A}\mathcal{C}%
\right) \left( \left( \nu A\right) \left( a\right) \right) .
\end{equation*}%
Since $\nu A:A\rightarrow A\otimes _{A}\mathcal{C}$ we can write $\left( \nu
A\right) \left( a\right) =\left( \nu A\right) \left( 1\right) a=\left(
1_{A}\otimes _{A}e\right) a=1_{A}\otimes _{A}ea$ for some $e\in \mathcal{C}.$
Thus
\begin{equation*}
\left( \nu M\right) \left( ma\right) =\left( f_{m}\otimes _{A}\mathcal{C}%
\right) \left( 1_{A}\otimes _{A}ea\right) =m\otimes _{A}ea
\end{equation*}%
and hence $\nu M:M\rightarrow M\otimes _{A}\mathcal{C}:m\mapsto m\otimes
_{A}e$ for every $M\in \mathcal{M}_{A}$. The proof of \cite[Theorem 3.3]{Br}
shows that $\epsilon \circ \nu =\mathrm{Id}$ is equivalent to require that $%
e\in \mathcal{C}^{A}\ $and $\varepsilon _{\mathcal{C}}\left( e\right) =1.$
On elements we can rewrite the equality $\nu \nu =L\eta R\circ \nu $ for
every $m\in M$ as%
\begin{equation*}
m\otimes _{A}e\otimes _{A}e=m\otimes _{A}\Delta _{\mathcal{C}}\left( e\right)
\end{equation*}%
which is equivalent to $\Delta _{\mathcal{C}}\left( e\right) =e\otimes
_{A}e. $
\end{invisible}
\end{proof}

\begin{remark}
Let $\mathcal{C}$ be an $A$-coring. We recall that, by \cite[Lemma 5.1]%
{Br}, if $A$ is a right $\mathcal{C}$-comodule via $\rho _{A}:A\rightarrow
A\otimes _{A}\mathcal{C}$, then $g=\rho _{A}\left( 1_{A}\right) $ is a
grouplike element of \ $\mathcal{C}$. Conversely if $g$ is a grouplike
element of \ $\mathcal{C}$, then $A$ is a right $\mathcal{C}$-comodule via $%
\rho _{A}:A\rightarrow A\otimes _{A}\mathcal{C}$ defined by $\rho _{A}\left(
a\right) =1_{A}\otimes _{A}\left( g\cdot a\right) $. Moreover, if $g$ is a
grouplike element of \ $\mathcal{C}$, then, by \cite[page 404]{Br}, $g$ is
an invariant element if and only if $A=A^{\text{co}\mathcal{C}}:=\left\{
a\in A\mid ag=ga\right\} $
\end{remark}

\section{Heavily separable ring homomorphisms}\label{sec:3}

Let $\varphi :R\rightarrow S$ be a ring homomorphism. Consider the induction
functor $\varphi ^{\ast }:=S\otimes _{R}(-): {} R$-$\mathrm{Mod}\to S$-$\mathrm{Mod}$.

\begin{proposition}\label{pro:inducfunhsep}
Let $\varphi :R\rightarrow S$ be a ring homomorphism. Then the induction
functor $\varphi ^{\ast }$ is h-separable if and only if there is a ring
homomorphism $E:S\rightarrow R$ such that $E\circ \varphi =\mathrm{Id}.$
\end{proposition}

\begin{proof}
By \cite[Theorem 27]{CMZ}, we know that $\varphi ^{\ast }$ is separable if
and only if there is a morphism of $R$-bimodules $E:S\rightarrow R$ such
that $E\left( 1_{S}\right) =1_{R}.$ Using $E$ one defines a natural
transformation $\gamma M:S\otimes _{R}M\rightarrow M:s\otimes _{R}m\mapsto
E\left( s\right) m,$ for every $M\in R$-$\mathrm{Mod,}$ such that $\gamma
\circ \eta =\mathrm{Id}$ where the unit $\eta $ is defined by $\eta
M:M\rightarrow S\otimes _{R}M:m\mapsto 1_{S}\otimes _{R}m.$ All natural
transformations $\gamma $ such that $\gamma \circ \eta =\mathrm{Id}$ are of
this form because $R$ is a generator in $R$-$\mathrm{Mod.}$

\begin{invisible}
For every $m\in M$ set $f_{m}:R\rightarrow M:r\mapsto rm.$ Then by
naturality of $\gamma $ we have%
\begin{equation*}
\left( \gamma M\right) \left( s\otimes _{R}m\right) =\left( \gamma M\right)
\left( S\otimes _{R}f_{m}\right) \left( s\otimes _{R}1_{R}\right)
=f_{m}\left( \gamma R\right) \left( s\otimes _{R}1_{R}\right) =\left( \gamma
R\right) \left( s\otimes _{R}1_{R}\right) \cdot m.
\end{equation*}%
Set $E\left( s\right) :=\left( \gamma R\right) \left( s\otimes
_{R}1_{R}\right) \in R.$ Then $\left( \gamma M\right) \left( s\otimes
_{R}m\right) =E\left( s\right) m.$ Now $m=\left( \gamma M\right) \left( \eta
M\right) \left( m\right) =\left( \gamma M\right) \left( 1_{S}\otimes
_{R}m\right) =E\left( 1_{S}\right) m$ so that $E\left( 1_{S}\right) =1_{R}.$
Moreover $E\left( rs\right) =\left( \gamma R\right) \left( rs\otimes
_{R}1_{R}\right) =\left( \gamma R\right) \left( r\left( s\otimes
_{R}1_{R}\right) \right) =r\left( \gamma R\right) \left( s\otimes
_{R}1_{R}\right) =rE\left( s\right) $ since $\gamma R\in R$-$\mathrm{Mod.}$
Furthermore $E\left( sr\right) =\left( \gamma R\right) \left( sr\otimes
_{R}1_{R}\right) =\left( \gamma R\right) \left( s\otimes _{R}r\right)
=\left( \gamma R\right) \left( s\otimes _{R}r\right) =E\left( s\right) r.$
Hence $E$ is $R$-bilinear.
\end{invisible}

Finally (\ref{form:gamma0}) rewrites as $E\left( x\right) E\left( y\right)
m=E\left( xy\right) m$ for every $x,y\in S$ and $m\in M,$ for every $M\in R$-%
$\mathrm{Mod.}$ Thus it is equivalent to ask that $E$ is multiplicative.
Summing up $\varphi ^{\ast }$ is h-separable if and only if there is a
morphism of $R$-bimodules $E:S\rightarrow R$ which is a ring homomorphism.
This is equivalent to ask that $E:S\rightarrow R$ is a ring homomorphism
such that $E\circ \varphi =\mathrm{Id}.$

\begin{invisible}
If $E$ is $R$-bilinear, then $E\left( \varphi \left( r\right) \right)
=E\left( r1_{S}\right) =rE\left( 1_{S}\right) =r1_{R}=r$ so that $E\circ
\varphi =\mathrm{Id}.$ Conversely, if $E\circ \varphi =\mathrm{Id}$ then $%
E\left( rsr^{\prime }\right) =E\left( \varphi \left( r\right) s\varphi
\left( r^{\prime }\right) \right) =E\varphi \left( r\right) E\left( s\right)
E\varphi \left( r^{\prime }\right) =rE\left( s\right) r^{\prime }.$
\end{invisible}
\end{proof}

It is well-known that the restriction of scalars functor $\varphi _{\ast }:S$-$\mathrm{Mod}\rightarrow R$-$\mathrm{Mod}$ is the right
adjoint of the induction functor $\varphi ^{\ast }.$ Moreover $\varphi _{\ast }$ is separable if and only if $S/R$ is separable
(see \cite[Proposition 1.3]{NVV}) if and only if it admits a separability
idempotent.

We are so lead to the following definition.

\begin{definitions}
\label{def:hsepmorph}1) $S/R$ is \textbf{h-separable} if the functor $%
\varphi _{\ast }:S$-$\mathrm{Mod}\rightarrow R$-$\mathrm{Mod}$ is
h-separable.

2) A \textbf{heavy separability idempotent} (\textbf{h-separability
idempotent} for short) of $S/R$ is an element $\sum_{i}a_{i}\otimes
_{R}b_{i}\in S\otimes _{R}S$ such that $\sum_{i}a_{i}\otimes _{R}b_{i}$ is a
separability idempotent, i.e.
\begin{equation}
\sum_{i}a_{i}b_{i}=1,\qquad \sum_{i}sa_{i}\otimes
_{R}b_{i}=\sum_{i}a_{i}\otimes _{R}b_{i}s\quad \text{ for every }s\in S,
\label{hvspdmp1}
\end{equation}%
which moreover fulfills
\begin{equation}
\sum_{i,j}a_{i}\otimes _{R}b_{i}a_{j}\otimes _{R}b_{j}=\sum_{i}a_{i}\otimes
_{R}1_{S}\otimes _{R}b_{i}.  \label{hvspdmp2}
\end{equation}
\end{definitions}

\begin{remark}
\label{rem:Sweedler}Note that a h-separability idempotent $%
e:=\sum_{i}a_{i}\otimes _{R}b_{i}$ is exactly a grouplike element in the
Sweedler's coring $\mathcal{C}:=S\otimes _{R}S$ such that $se=es$ for every $%
s\in S$ i.e. which is invariant. Note that $1_{S}\otimes _{R}1_{S}$ is
always a grouplike element in $\mathcal{C}$ but it is not invariant in
general.
\end{remark}

\begin{proposition}
\label{pro:S/R}$S/R$ is h-separable if and only if it has a h-separability
idempotent.
\end{proposition}

\begin{proof}
We observed that $\varphi _{\ast }$ is the right adjoint of the induction functor $%
\varphi ^{\ast }:=S\otimes _{R}\left( -\right) $. Recall that $S/R$ is separable if and only if the map $%
S\otimes _{R}S\rightarrow S$ splits as an $S$-bimodule map. The splitting is
uniquely determined by a so-called separability idempotent i.e. an element $%
\sum_{i}a_{i}\otimes _{R}b_{i}\in S\otimes _{R}S$ such that \eqref{hvspdmp1} hold. Using this
element we can define $\delta :\mathrm{Id}\rightarrow \varphi ^{\ast
}\varphi _{\ast }$ such that $\epsilon \circ \delta =\mathrm{Id}$ by $\delta
M:M\rightarrow S\otimes _{R}M:m\mapsto \sum_{i}a_{i}\otimes _{R}b_{i}m.$
This natural transformation satisfies (\ref{form:delta0}) if and only if%
\begin{equation*}
\left( S\otimes _{R}\delta M\right) \circ \delta M=\left( S\otimes _{R}\eta
M\right) \circ \delta M.
\end{equation*}%
Let us compute separately the two terms of this equality on any $m\in M,$
\begin{gather*}
\left( S\otimes _{R}\delta M\right) \left( \delta M\right) \left( m\right)
=\sum_{i}a_{i}\otimes _{R}\left( \delta M\right) \left( b_{i}m\right)
=\sum_{i,j}a_{i}\otimes _{R}a_{j}\otimes
_{R}b_{j}b_{i}m\overset{\eqref{hvspdmp1}}{=}\sum_{i,j}a_{i}\otimes _{R}b_{i}a_{j}\otimes _{R}b_{j}m, \\
\left( S\otimes _{R}\eta M\right) \left( \delta M\right) \left( m\right)
=\sum_{i}a_{i}\otimes _{R}\left( \eta M\right) \left( b_{i}m\right)
=\sum_{i}a_{i}\otimes _{R}1_{S}\otimes _{R}b_{i}m.
\end{gather*}%
Thus $\delta $ satisfies (\ref{form:delta0}) if and only if \eqref{hvspdmp2}
holds true.
\end{proof}

\begin{remark}
Let $\varphi :R\rightarrow S$ be a ring homomorphism and let $\mathcal{C}%
:=S\otimes _{R}S$ be the Sweedler coring. In view of Theorem \ref{thm:indu}
and Remark \ref{rem:Sweedler} we obtain that $S/R$ is h-separable (i.e. the
functor $\varphi _{\ast }:S$-$\mathrm{Mod}\rightarrow R$-$\mathrm{Mod}$ is
h-separable) if and only if the induction functor $R:=\left( -\right)
\otimes _{S}\mathcal{C}:\mathrm{Mod}$-$S\rightarrow \mathcal{M}^{\mathcal{C}}
$ is h-separable. Note that here $\mathcal{M}^{\mathcal{C}}$ is isomorphic
(see e.g. \cite[page 252-253.]{Brz-Wisb}) to the category $\mathrm{Desc}%
\left( S/R\right) $ of descent data associated to the ring homomorphism $%
\varphi .$
\end{remark}

\begin{corollary}
\label{coro:hsepsurj}Let $\varphi :R\rightarrow S$ and $\psi :S\rightarrow T$
be ring homomorphisms.

\begin{itemize}
\item[1)] If $T/S$ and $S/R$ are h-separable so is $T/R$.

\item[2)] If $T/R$ is h-separable so is $T/S$.

\item[3)] If $S/R$ is h-separable then $T/S$ is h-separable if and only if
so is $T/R.$
\end{itemize}
\end{corollary}

\begin{proof}
It follows by Definition \ref{def:hsepmorph} and Lemma \ref{lem:compheav}.
\begin{invisible}
1) $T/S$ and $S/R$ are h-separable $\Leftrightarrow \psi _{\ast }$ and $%
\varphi _{\ast }$ are h-separable. By Lemma \ref{lem:compheav}, this
implies that $\varphi _{\ast }\circ \psi _{\ast }$ h-separable. Thus $\left(
\psi \varphi \right) _{\ast }$ is h-separable i.e. so is $T/R.$

2) $T/R$ is h-separable means that the functor $\left( \psi \varphi \right)
_{\ast }$ is h-separable i.e. $\varphi _{\ast }\circ \psi _{\ast }$
h-separable. By Lemma \ref{lem:compheav}, this implies that $\psi _{\ast }$
is h-separable i.e. $T/S$ is h-separable.
\end{invisible}
\end{proof}

\begin{lemma}
\label{lem:ringepim}Let $\varphi :R\rightarrow S$ be a ring homomorphism.
The following are equivalent.

\begin{enumerate}
\item The map $\varphi $ is a ring epimorphism (i.e. an epimorphism in the
category of rings);

\item the multiplication $m:S\otimes _{R}S\rightarrow S$ is invertible;

\item $1_{S}\otimes _{R}1_{S}$ is a separability idempotent for $S/R;$

\item $1_{S}\otimes _{R}1_{S}$ is a h-separability idempotent for $S/R.$
\end{enumerate}

If these equivalent conditions hold true then $S/R$ is h-separable.

Moreover $1_{S}\otimes _{R}1_{S}$ is the unique separability idempotent for $%
S/R.$
\end{lemma}

\begin{proof}
$\left( 1\right) \Leftrightarrow \left( 2\right) $ follows by \cite[%
Proposition XI.1.2 page 225]{Stenstroem}.

$\left( 1\right) \Leftrightarrow \left( 3\right) $ follows by \cite[%
Proposition XI.1.1 page 226]{Stenstroem}.

\begin{invisible}
If $m$ is invertible, from $m\left( s\otimes _{R}1_{S}\right) =s=m\left(
1_{S}\otimes _{R}s\right) $ and the injectivity of $m$ we deduce the claimed
equality. Conversely if the equality holds let $h:S\rightarrow S\otimes
_{R}S:s\mapsto s\otimes _{R}1_{S}.$ Then $m\circ h=\mathrm{Id}$ and
\begin{equation*}
\left( h\circ m\right) \left( s^{\prime }\otimes _{R}s\right) =s^{\prime
}s\otimes _{R}1_{S}=\left( m\otimes _{R}S\right) \left( s^{\prime }\otimes
_{R}s\otimes _{R}1_{S}\right) =\left( m\otimes _{R}S\right) \left( s^{\prime
}\otimes _{R}1_{S}\otimes _{R}s\right) =s^{\prime }\otimes _{R}s
\end{equation*}%
and hence $h\circ m=\mathrm{Id}.$ Hence $m$ is invertible.
\end{invisible}

$\left( 3\right) \Leftrightarrow \left( 4\right) $ depends on the fact that $%
1_{S}\otimes _{R}1_{S}$ always fulfills (\ref{hvspdmp2}).

By Proposition \ref{pro:S/R}, $\left( 4\right) $ implies that $S/R$ is
h-separable.

Let us check the last part of the statement. If $\sum_{i}a_{i}\otimes
_{R}b_{i}$ is another separability idempotent, we get
\begin{equation*}
\sum_{i}a_{i}\otimes _{R}b_{i}=\sum_{i}a_{i}1_{S}\otimes
_{R}1_{S}b_{i}=1_{S}\otimes _{R}1_{S}\sum_{i}a_{i}b_{i}\overset{\eqref{hvspdmp1}}{=}1_{S}\otimes
_{R}1_{S}.
\end{equation*}
\end{proof}

\begin{example}
\label{ex:multiso}We now give examples of ring epimorphisms $\varphi
:R\rightarrow S$.

1) Let $S$ be a multiplicative closed subset of a commutative ring $R.$ Then
the canonical map $\varphi :R\rightarrow S^{-1}R$ is a ring epimorphism,
c.f. \cite[Proposition 3.1]{AtMac}. More generally we can consider a perfect
right localization of $R$ as in \cite[page 229]{Stenstroem}.

\begin{invisible}
The case $\left( S^{-1}R\right) /R$ where $S$ is a multiplicative closed
subset of a commutative ring $R.$ In fact
\begin{equation*}
\frac{r}{s}\otimes _{R}1=\frac{r}{s}\otimes _{R}\frac{s}{s}=\frac{s}{s}%
\otimes _{R}\frac{r}{s}=1\otimes _{R}\frac{e}{s}.
\end{equation*}%
In particular $\mathbb{Q}/\mathbb{Z}$ is of this form.
\end{invisible}

2) Consider the ring of matrices $\mathrm{M}_{n}\left( R\right) $ and the
ring $\mathrm{T}_{n}\left( R\right) $ of $n\times n$ upper triangular
matrices over a ring $R.$ Then the inclusion $\varphi :\mathrm{T}_{n}\left(
R\right) \rightarrow \mathrm{M}_{n}\left( R\right) $ is a ring epimorphism.
In fact, given ring homomorphisms $\alpha ,\beta :\mathrm{M}_{n}\left(
R\right) \rightarrow B$ that coincide on $\mathrm{T}_{n}\left( R\right) $
then they coincide on all matrices. To see this we first check that $\alpha
\left( E_{ij}\right) =\beta \left( E_{ij}\right) $ for all $i>j$,%
\begin{eqnarray*}
\alpha \left( E_{ij}\right) &=&\alpha \left( E_{ij}E_{jj}\right) =\alpha
\left( E_{ij}\right) \alpha \left( E_{jj}\right) =\alpha \left(
E_{ij}\right) \beta \left( E_{jj}\right) =\alpha \left( E_{ij}\right) \beta
\left( E_{ji}E_{ij}\right) \\
&=&\alpha \left( E_{ij}\right) \beta \left( E_{ji}\right) \beta \left(
E_{ij}\right) =\alpha \left( E_{ij}\right) \alpha \left( E_{ji}\right) \beta
\left( E_{ij}\right) =\alpha \left( E_{ij}E_{ji}\right) \beta \left(
E_{ij}\right) \\
&=&\alpha \left( E_{ii}\right) \beta \left( E_{ij}\right) =\beta \left(
E_{ii}\right) \beta \left( E_{ij}\right) =\beta \left( E_{ii}E_{ij}\right)
=\beta \left( E_{ij}\right) .
\end{eqnarray*}%
Thus $\alpha \left( E_{ij}\right) =\beta \left( E_{ij}\right) $ for every $%
i,j.$ Now, given $r\in R$ we have%
\begin{equation*}
\alpha \left( rE_{ij}\right) =\alpha \left( rE_{ii}E_{ij}\right) =\alpha
\left( rE_{ii}\right) \alpha \left( E_{ij}\right) =\beta \left(
rE_{ii}\right) \beta \left( E_{ij}\right) =\beta \left( rE_{ii}E_{ij}\right)
=\beta \left( rE_{ij}\right) .
\end{equation*}

As a consequence $\alpha \left( M\right) =\beta \left( M\right) $ for all $%
M\in \mathrm{M}_{n}\left( R\right) $ as desired.

3) Any surjective ring homomorphism $\varphi :R\rightarrow S$ is trivially a
ring epimorphism.
\end{example}

\begin{remark}
A kind of dual to Lemma \ref{lem:ringepim}, establishes that a $\Bbbk $%
-coalgebra homomorphism $\varphi :C\rightarrow D$ is a coalgebra epimorphism
if and only if the induced functor $\mathcal{M}^{C}\rightarrow \mathcal{M}%
^{D}$ is full, see \cite[Theorem 3.5]{NT}. Since this functor is always
faithful, by Lemma \ref{lem:ffhsep} it is in this case h-separable.
\end{remark}

\begin{proposition}
\label{pro:hsepext}Let $\varphi :R\rightarrow S$ be a ring homomorphism.
Then $S/R$ is h-separable if and only if $S/\varphi \left( R\right) $ is
h-separable.
\end{proposition}

\begin{proof}
Write $\varphi =i\circ \overline{\varphi }$ where $i:\varphi \left( R\right)
\rightarrow S$ is the canonical inclusion and $\overline{\varphi }%
:R\rightarrow \varphi \left( R\right) $ is the corestriction of $\varphi $
to its image $\varphi \left( R\right) .$ By Lemma \ref{lem:ringepim}, we
have that $\varphi \left( R\right) /R$ is h-separable.

By Corollary \ref{coro:hsepsurj}, $S/R$ is h-separable if and only if $%
S/\varphi \left( R\right) $ is h-separable.
\end{proof}

It is well-known that the ring of matrices is separable, see e.g. \cite[%
Example II, page 41]{DI}. In the following result we show that however it is
never h-separable.

\begin{lemma}
\label{lem:matrix}Let $R$ be a ring and $S:=\mathrm{M}_{n}\left(
R\right) $. If $S/R$ is h-separable, then either $n=1$ or $R=0.$
\end{lemma}

\begin{proof}
By Proposition \ref{pro:S/R}, $S/R$ admits a h-separability idempotent $%
e=\sum_{i,j,s,t}r_{i,j}^{s,t}E_{i,j}\otimes _{R}E_{s,t}$ where $%
r_{i,j}^{s,t}\in R$ and $E_{i,j}$ is the canonical matrix having $1$ in the
entry $\left( i,j\right) $ and zero elsewhere ($e$ can be written in the
given form since the tensor product is over $R$ and the $E_{i,j}$'s are $R$%
-invariant). Then%
\begin{eqnarray*}
E_{a,b}e &=&\sum_{i,j,s,t}E_{a,b}r_{i,j}^{s,t}E_{i,j}\otimes
_{R}E_{s,t}=\sum_{i,j,s,t}r_{i,j}^{s,t}E_{a,b}E_{i,j}\otimes
_{R}E_{s,t}=\sum_{j,s,t}r_{b,j}^{s,t}E_{a,j}\otimes _{R}E_{s,t}, \\
eE_{a,b} &=&\sum_{i,j,s,t}r_{i,j}^{s,t}E_{i,j}\otimes
_{R}E_{s,t}E_{a,b}=\sum_{i,j,s}r_{i,j}^{s,a}E_{i,j}\otimes _{R}E_{s,b}.
\end{eqnarray*}%
Note that the elements $E_{i,j}\otimes _{R}E_{i^{\prime },j^{\prime }}$'s
form a basis of $S\otimes _{R}S$ as a free left $R$-module. Thus, the
equalities $E_{a,b}e=eE_{a,b}$ implies $r_{b,j}^{s,t}=0$ for $t\neq b.$
Moreover if $t=b$ then $r_{b,j}^{s,t}=r_{a,j}^{s,a}.$ Thus $%
r_{b,j}^{s,t}=\delta _{t,b}r_{a,j}^{s,a}$ for every $a,b,j,s.$ Thus, if we
set $r_{j}^{s}:=r_{1,j}^{s,1},$ we obtain $r_{b,j}^{s,t}=\delta
_{t,b}r_{1,j}^{s,1}=\delta _{t,b}r_{j}^{s}.$ We can now rewrite
\begin{equation*}
e=\sum_{i,j,s,t}r_{i,j}^{s,t}E_{i,j}\otimes
_{R}E_{s,t}=\sum_{i,j,s}r_{j}^{s}E_{i,j}\otimes
_{R}E_{s,i}=\sum_{j,s}r_{j}^{s}\sum_{i}E_{i,j}\otimes _{R}E_{s,i}.
\end{equation*}%
Now%
\begin{equation*}
\sum_{i}E_{i,i}=1=m\left( e\right)
=\sum_{j,s}r_{j}^{s}\sum_{i}E_{i,j}E_{s,i}=\sum_{j}r_{j}^{j}\sum_{i}E_{i,i}
\end{equation*}%
so that
\begin{equation}\label{form:rjj}
\sum_{j}r_{j}^{j}=1_{R}.
\end{equation}%
From (\ref{hvspdmp2}) and the fact that the $E_{i,j}$'s are $R$-invariant we
deduce that%
\begin{equation*}
\sum_{j,s}\sum_{j^{\prime },s^{\prime }}r_{j}^{s}r_{j^{\prime }}^{s^{\prime
}}\sum_{i}\sum_{i^{\prime }}E_{i,j}\otimes _{R}E_{s,i}E_{i^{\prime
},j^{\prime }}\otimes _{R}E_{s^{\prime },i^{\prime
}}=\sum_{j,s}r_{j}^{s}\sum_{i}E_{i,j}\otimes _{R}1\otimes _{R}E_{s,i}
\end{equation*}%
i.e.%
\begin{equation*}
\sum_{j,s}\sum_{j^{\prime },s^{\prime }}r_{j}^{s}r_{j^{\prime }}^{s^{\prime
}}\sum_{i}E_{i,j}\otimes _{R}E_{s,j^{\prime }}\otimes _{R}E_{s^{\prime
},i}=\sum_{j,s}r_{j}^{s}\sum_{j^{\prime },i}E_{i,j}\otimes _{R}E_{j^{\prime
},j^{\prime }}\otimes _{R}E_{s,i}.
\end{equation*}%
Equivalently for all $i,j$ we have
\begin{equation*}
\sum_{s}\sum_{j^{\prime },s^{\prime }}r_{j}^{s}r_{j^{\prime }}^{s^{\prime
}}E_{s,j^{\prime }}\otimes _{R}E_{s^{\prime
},i}=\sum_{s}r_{j}^{s}\sum_{j^{\prime }}E_{j^{\prime },j^{\prime }}\otimes
_{R}E_{s,i}
\end{equation*}%
i.e., replacing $s$ with $s^{\prime }$ in the second term,%
\begin{equation*}
\sum_{s}\sum_{j^{\prime },s^{\prime }}r_{j}^{s}r_{j^{\prime }}^{s^{\prime
}}E_{s,j^{\prime }}\otimes _{R}E_{s^{\prime },i}=\sum_{s^{\prime
}}r_{j}^{s^{\prime }}\sum_{j^{\prime }}E_{j^{\prime },j^{\prime }}\otimes
_{R}E_{s^{\prime },i}.
\end{equation*}%
Thus for every $s^{\prime },i,j$%
\begin{equation*}
\sum_{s}\sum_{j^{\prime }}r_{j}^{s}r_{j^{\prime }}^{s^{\prime
}}E_{s,j^{\prime }}=r_{j}^{s^{\prime }}\sum_{j^{\prime }}E_{j^{\prime
},j^{\prime }}.
\end{equation*}%
From this equality we deduce $r_{j}^{s}r_{j^{\prime }}^{s^{\prime }}=\delta
_{s,j^{\prime }}r_{j}^{s^{\prime }}$ for all $s,j,s^{\prime },j^{\prime }.$
Now%
\begin{equation*}
r_{j}^{s^{\prime }}=\sum_{j^{\prime }}\delta _{s,j^{\prime
}}r_{j}^{s^{\prime }}=\sum_{j^{\prime }}r_{j}^{s}r_{j^{\prime }}^{j^{\prime
}}=r_{j}^{s}\sum_{j^{\prime }}r_{j^{\prime }}^{j^{\prime }}\overset{\eqref{form:rjj}}{=}r_{j}^{s}\cdot
1_{R}=r_{j}^{s}
\end{equation*}%
so that we can set $r_{j}:=r_{j}^{1}$ and we get $r_{j}^{s}=r_{j}$ for each $%
s,j.$ Hence the equality $r_{j}^{s}r_{j^{\prime }}^{s^{\prime }}=\delta
_{s,j^{\prime }}r_{j}^{s^{\prime }}$ rewrites as $r_{j}r_{j^{\prime
}}=\delta _{s,j^{\prime }}r_{j}$ for all $s,j,j^{\prime }.$

If $n\geq 2,$ then for every $j^{\prime }$ there is always $s\neq j^{\prime }\
so$ that we obtain $r_{j}r_{j^{\prime }}=0$ for all $j,j^{\prime }.$ Now $%
0=\sum_{j}\sum_{j^{\prime }}r_{j}r_{j^{\prime
}}=\sum_{j}r_{j}\sum_{j^{\prime }}r_{j^{\prime }}=1_{R}\cdot 1_{R}=1_{R}$, a
contradiction.
\end{proof}

\subsection{Heavily separable algebras}

\begin{claim}
Let $R$ be a commutative ring, let $S$ be a ring and let $Z\left( S\right) $
be its center. We recall that a $S$ is said to be an $R$-algebra, or that $S$
is an algebra over $R$, if there is a unital ring homomorphism $\varphi
:R\rightarrow S$ such that $\varphi \left( R\right) \subseteq Z\left(
S\right) $. In this case we set%
\begin{equation*}
r\cdot s=\varphi \left( r\right) \cdot s\text{ for every }r\in R\text{ and }%
s\in S.
\end{equation*}%
Since ${Im}\left( \varphi \right) \subseteq Z\left( S\right) ,$ we have $%
r\cdot s=s\cdot r$ for every $r\in R$ and $s\in S$ and
\begin{equation*}
r\cdot 1_{S}=\varphi \left( r\right) \cdot 1_{S}=\varphi \left( r\right) \cdot
\varphi \left( 1_{R}\right) =\varphi \left( r\cdot 1_{R}\right) =\varphi \left(
r\right) \text{ for every }r\in R\text{ so that }R1_{S}={Im}\left( \varphi
\right) \subseteq Z\left( S\right) .
\end{equation*}
\end{claim}

\begin{theorem}
\label{thm:alg}Let $S$ be and $R$-algebra. Then $S/R$ is h-separable if and
only if the canonical map $\varphi :R\rightarrow S$ is a ring epimorphism.
Moreover if one of these conditions holds, then $S$ is commutative.
\end{theorem}

\begin{proof}
$\left( \Rightarrow \right) .$ Let $\sum_{i}a_{i}\otimes _{R}b_{i}$ be an
h-separability idempotent. Since $\varphi \left( R\right) \subseteq Z\left(
S\right) ,$ we get that the map $\tau :A\otimes _{R}A\rightarrow A\otimes
_{R}A,\tau \left( a\otimes _{R}b\right) =b\otimes _{R}a,$ is well-defined
and left $R$-linear. Hence we can apply $A\otimes _{R}\tau $ on both sides
of (\ref{hvspdmp2}) to get $\sum_{i,j}a_{i}\otimes _{R}b_{j}\otimes
_{R}b_{i}a_{j}=\sum_{i}a_{i}\otimes _{R}b_{i}\otimes _{R}1_{S}.$ By
multiplying, we obtain $\sum_{i,j}a_{i}b_{j}\otimes
_{R}b_{i}a_{j}=\sum_{i}a_{i}b_{i}\otimes _{R}1_{S}.$ By (\ref{hvspdmp1}), we
get
\begin{equation}
\sum_{i,j}a_{i}b_{j}\otimes _{R}b_{i}a_{j}=1_{S}\otimes _{R}1_{S}.
\label{form:star}
\end{equation}%
By (\ref{hvspdmp1}), we get that $\sum_{t}a_{t}sb_{t}\in Z\left( S\right) $, for all $s\in S$.
Using this fact we have%
\begin{eqnarray*}
s &=&1_{S}\cdot 1_{S}\cdot s\overset{(\ref{form:star})}{=}%
\sum_{i,j}a_{i}b_{j}b_{i}a_{j}s=\sum_{i,j}a_{i}\left( b_{j}\right)
b_{i}\left( a_{j}\right) s\left( 1_{S}\right) \\
&\overset{(\ref{hvspdmp2})}{=}&\sum_{i,j,t}a_{i}b_{j}b_{i}\left(
a_{t}sb_{t}\right) a_{j}=\sum_{i,j,t}a_{i}b_{j}b_{i}a_{j}\left(
a_{t}sb_{t}\right) \overset{(\ref{form:star})}{=}\sum_{t}a_{t}sb_{t}\in
Z\left( S\right) .
\end{eqnarray*}%
We have so proved that $S\subseteq Z\left( S\right) $ and hence $S$ is
commutative. Now, we compute%
\begin{equation*}
\sum_{i}a_{i}\otimes _{R}b_{i}\overset{(\ref{hvspdmp1})}{=}%
\sum_{i,j}a_{i}a_{j}b_{j}\otimes _{R}b_{i}\overset{S=Z\left( S\right) }{=}%
\sum_{i,j}a_{j}a_{i}b_{j}\otimes _{R}b_{i}\overset{(\ref{hvspdmp1})}{=}%
\sum_{i,j}a_{i}b_{j}\otimes _{R}b_{i}a_{j}
\end{equation*}%
so that $\sum_{i}a_{i}\otimes _{R}b_{i}=1_{S}\otimes _{R}1_{S}$ by (\ref%
{form:star}). We conclude by Lemma \ref{lem:ringepim}.

$\left( \Leftarrow \right) $ It follows by Lemma \ref{lem:ringepim}.
\end{proof}

The following result establishes that there is no non-trivial h-separable
algebra over a field $\Bbbk $.

\begin{proposition}
\label{Pro:hsepoverfield}Let $A$ be a h-separable algebra over a field $%
\Bbbk $. Then either $A=\Bbbk $ or $A=0.$
\end{proposition}

\begin{proof}
By Theorem \ref{thm:alg}, the unit $u:\Bbbk \rightarrow A$ is a ring
epimorphism. By Lemma \ref{lem:ringepim}, we have that $A\otimes _{\Bbbk
}A\cong A$ via multiplication. Since $A$ is h-separable over $\Bbbk $ it is
in particular separable over $\Bbbk .$ By \cite[page 184]{Pi1}, the
separable $\Bbbk $-algebra $A$ is finite-dimensional. Thus, from $A\otimes
_{\Bbbk }A\cong A$ we deduce that $A$ has either dimensional one or
zero over $\Bbbk $.
\end{proof}

\begin{example}
$\mathbb{C}/\mathbb{R}$ is separable but not h-separable. In fact, by
Proposition \ref{Pro:hsepoverfield}, $\mathbb{C}/\mathbb{R}$ is not
h-separable. On the other hand $e=\frac{1}{2}\left( 1\otimes 1-i\otimes
i\right) $ is a separability idempotent (it is the only possible one). It is
clear that $e$ is not a h-separability idempotent.
\end{example}

\begin{invisible}
We now that $\mathbb{C}/\mathbb{R}$ is classically separable because it is a
finite extension. Let us see if it is h-separable as well. Let $%
e=a_{11}1\otimes 1+a_{12}1\otimes i+a_{21}i\otimes 1+a_{22}i\otimes i$ be a
h-separability idempotent. Then $1=m\left( e\right)
=a_{11}-a_{22}+a_{12}i+a_{21}i$ so that $a_{11}=a_{22}+1$ and $%
a_{12}+a_{21}=0.$ Moreover%
\begin{eqnarray*}
ie &=&a_{11}i\otimes 1+a_{12}i\otimes i-a_{21}1\otimes 1-a_{22}1\otimes i, \\
ei &=&a_{11}1\otimes i-a_{12}1\otimes 1+a_{21}i\otimes i-a_{22}i\otimes 1.
\end{eqnarray*}%
Since $ie=ei$ we get $a_{11}=-a_{22},a_{12}=a_{21}.$ Thus $a_{12}=a_{21}=0,$
$a_{22}=-\frac{1}{2}$ and $a_{11}=\frac{1}{2}$ so that
\begin{equation*}
e=\frac{1}{2}\left( 1\otimes 1-i\otimes i\right) .
\end{equation*}

Moreover the condition $\left( \ref{hvspdmp2}\right) $ rewrites as%
\begin{equation*}
\frac{1}{4}\left( 1\otimes 1\otimes 1-1\otimes i\otimes i-i\otimes i\otimes
1-i\otimes 1\otimes i\right) =\frac{1}{2}\left( 1\otimes 1\otimes 1-i\otimes
1\otimes i\right)
\end{equation*}%
a contradiction. Thus $\mathbb{C}/\mathbb{R}$ is not h-separable.
\end{invisible}

\begin{remark}
Let $\Bbbk $ be a field. Set $A=B=\Bbbk ,R=\Bbbk \times \Bbbk =S.$ Then $A$
and $B$ are $R$-algebras and $S=A\times B$ is their product in the category
of $R$-algebras. Moreover $S/R$ is h-separable as $S=R.$ Hence the product
of $R$-algebras may be h-separable.
\end{remark}

\begin{lemma}
\label{lem:prodinsep}Let $A$ and $B$ be $R$-algebras and let $S=A\times B$
be their product in the category of $R$-algebras. Set $e_{1}:=\left(
1_{A},0_{B}\right) \in S$ and $e_{2}:=\left( 0_{A},1_{B}\right) \in S.$ The
following are equivalent.

\begin{itemize}
\item[$\left( i\right) $] $S/R$ is h-separable.

\item[$\left( ii\right) $] $A/R$ and $B/R$ are h-separable and $e_{1}\otimes
_{R}e_{2}=0=e_{2}\otimes _{R}e_{1}.$
\end{itemize}
\end{lemma}

\begin{proof}
First, by Theorem \ref{thm:alg} and Lemma \ref{lem:ringepim}, the conditions
$\left( i\right) $ and $\left( ii\right) $ can be replaced respectively by

\begin{itemize}
\item $1_{S}\otimes _{R}1_{S}$ is a separability idempotent of $S/R$

\item $1_{A}\otimes _{R}1_{A}$ and $1_{B}\otimes _{R}1_{B}$ are separability
idempotents of $A/R$ and $B/R$ respectively and $e_{1}\otimes
_{R}e_{2}=0=e_{2}\otimes _{R}e_{1}.$
\end{itemize}

Note also that, if the first condition holds, then, for $i\neq j$ we get
\begin{equation*}
e_{i}\otimes _{R}e_{j}=e_{i}1_{S}\otimes _{R}1_{S}e_{j}=1_{S}\otimes
_{R}1_{S}e_{i}e_{j}=0
\end{equation*}

so that $e_{1}\otimes _{R}e_{2}=0=e_{2}\otimes _{R}e_{1}.$ Thus the latter
condition can be assumed to hold.

Denote by $p_{A}:S\rightarrow A$ and $p_{B}:S\rightarrow B$ the canonical
projections.

Since $1_{S}=e_{1}+e_{2}$ and $e_{1}\otimes _{R}e_{2}=0=e_{2}\otimes
_{R}e_{1},$ we get that $1_{S}\otimes _{R}1_{S}=e_{1}\otimes
_{R}e_{1}+e_{2}\otimes _{R}e_{2}$. For every $s\in S$ we have%
\begin{eqnarray*}
s1_{S}\otimes _{R}1_{S} &=&se_{1}\otimes _{R}e_{1}+se_{2}\otimes _{R}e_{2} \\
&=&s\left( 1_{A},0_{B}\right) \otimes _{R}\left( 1_{A},0_{B}\right) +s\left(
0_{A},1_{B}\right) \otimes _{R}\left( 0_{A},1_{B}\right) \\
&=&\left( p_{A}\left( s\right) 1_{A},0_{B}\right) \otimes _{R}\left(
1_{A},0_{B}\right) +\left( 0_{A},p_{B}\left( s\right) 1_{B}\right) \otimes
_{R}\left( 0_{A},1_{B}\right) \\
&=&\left( i_{A}\otimes _{R}i_{A}\right) \left( p_{A}\left( s\right)
1_{A}\otimes _{R}1_{A}\right) +\left( i_{B}\otimes _{R}i_{B}\right) \left(
p_{B}\left( s\right) 1_{B}\otimes _{R}1_{B}\right)
\end{eqnarray*}%
Similarly $1_{S}\otimes _{R}1_{S}s=\left( i_{A}\otimes _{R}i_{A}\right)
\left( 1_{A}\otimes _{R}1_{A}p_{A}\left( s\right) \right) +\left(
i_{B}\otimes _{R}i_{B}\right) \left( 1_{B}\otimes _{R}1_{B}p_{B}\left(
s\right) \right) .$

\begin{invisible}
Here is the computation
\begin{eqnarray*}
&=&\left( i_{A}\otimes _{R}i_{A}\right) \left( 1_{A}\otimes
_{R}1_{A}p_{A}\left( s\right) \right) +\left( i_{B}\otimes _{R}i_{B}\right)
\left( 1_{B}\otimes _{R}1_{B}p_{B}\left( s\right) \right) \\
&=&\left( 1_{A},0_{B}\right) \otimes _{R}\left( 1_{A}p_{A}\left( s\right)
,0_{B}\right) +\left( 0_{A},1_{B}\right) \otimes _{R}\left(
0_{A},1_{B}p_{B}\left( s\right) \right) \\
&=&\left( 1_{A},0_{B}\right) \otimes _{R}\left( 1_{A},0_{B}\right) s+\left(
0_{A},1_{B}\right) \otimes _{R}\left( 0_{A},1_{B}\right) s \\
&=&e_{1}\otimes _{R}e_{1}s+e_{2}\otimes _{R}e_{2}s=1_{S}\otimes _{R}1_{S}s
\end{eqnarray*}
\end{invisible}

As a consequence, for every $s\in S$
\begin{equation*}
s1_{S}\otimes _{R}1_{S}=1_{S}\otimes _{R}1_{S}s\Longleftrightarrow
\end{equation*}%
\begin{eqnarray*}
&&\left( i_{A}\otimes _{R}i_{A}\right) \left( p_{A}\left( s\right)
1_{A}\otimes _{R}1_{A}\right) +\left( i_{B}\otimes _{R}i_{B}\right) \left(
p_{B}\left( s\right) 1_{B}\otimes _{R}1_{B}\right) \\
&=&\left( i_{A}\otimes _{R}i_{A}\right) \left( 1_{A}\otimes
_{R}1_{A}p_{A}\left( s\right) \right) +\left( i_{B}\otimes _{R}i_{B}\right)
\left( 1_{B}\otimes _{R}1_{B}p_{B}\left( s\right) \right) \Longleftrightarrow
\end{eqnarray*}%
\begin{equation*}
p_{A}\left( s\right) 1_{A}\otimes _{R}1_{A}=1_{A}\otimes
_{R}1_{A}p_{A}\left( s\right) \text{\qquad and \qquad }p_{B}\left( s\right)
1_{B}\otimes _{R}1_{B}=1_{B}\otimes _{R}1_{B}p_{B}\left( s\right) .
\end{equation*}

Since $p_{A}$ and $p_{B}$ are surjective, we get that to require that $%
s1_{S}\otimes _{R}1_{S}=1_{S}\otimes _{R}1_{S}s$ for every $s\in S$ is
equivalent to require that%
\begin{equation*}
a1_{A}\otimes _{R}1_{A}=1_{A}\otimes _{R}1_{A}a\text{\qquad and \qquad }%
b1_{B}\otimes _{R}1_{B}=1_{B}\otimes _{R}1_{B}b
\end{equation*}%
for every $a\in A,b\in B.$ We have so proved that $1_{S}\otimes _{R}1_{S}$
is a separability idempotent of $S/R$ if and only if $1_{A}\otimes _{R}1_{A}$
and $1_{B}\otimes _{R}1_{B}$ are separability idempotents of $A/R$ and $B/R$
under the assumption $e_{1}\otimes _{R}e_{2}=0=e_{2}\otimes _{R}e_{1}.$
\end{proof}

The following result is similar to \cite[Corollary 1.7 page 44]{DI}.

\begin{lemma}
\label{Lem:otalg}Let $R$ be a commutative ring. Let $A$ and $B$ be $R$%
-algebras. Then, if $B/R$ is h-separable, so is $\left( A\otimes
_{R}B\right) /A$. As a consequence if both $A/R$ and $B/R$ are h-separable,
so is $\left( A\otimes _{R}B\right) /R$ .
\end{lemma}

\begin{proof}
Since $B/R$ is h-separable, by Theorem \ref{thm:alg}, we have that the unit $%
u_{B}:R\rightarrow B$ is a ring epimorphism. By Lemma \ref{lem:ringepim}
this means that $1_{B}\otimes _{R}1_{B}$ is a separability idempotent. Thus
also $\left( 1_{A}\otimes _{R}1_{B}\right) \otimes _{A}\left( 1_{A}\otimes
_{R}1_{B}\right) $ is a separability idempotent.

\begin{invisible}
Let us check it:

$\left( a\otimes _{R}b\right) \left( 1_{A}\otimes _{R}1_{B}\right) \otimes
_{A}\left( 1_{A}\otimes _{R}1_{B}\right) =\left( 1_{A}\otimes
_{R}b1_{B}\right) \left( a\otimes _{R}1_{B}\right) \otimes _{A}\left(
1_{A}\otimes _{R}1_{B}\right) $

$=\left( 1_{A}\otimes _{R}1_{B}\right) \otimes _{A}\left( a\otimes
_{R}1_{B}\right) \left( 1_{A}\otimes _{R}1_{B}b\right) =\left( 1_{A}\otimes
_{R}1_{B}\right) \otimes _{A}\left( 1_{A}\otimes _{R}1_{B}\right) \left(
a\otimes _{R}b\right) .$
\end{invisible}

As a consequence also $A\otimes _{R}u_{B}:A\otimes _{R}R\rightarrow A\otimes
_{R}B$ is a ring epimorphism by the same lemma. If $A/R$ is h-separable,
then the unit $u_{A}:R\rightarrow A$ is a ring epimorphism too. Thus the
composition
\begin{equation*}
\xymatrixcolsep{0.7cm}\xymatrixrowsep{0.7cm}
\xymatrix{R\ar[r]^-{u_A}&A\cong A\otimes _{R}R\ar[rr]^-{A\otimes
_{R}u_{B}}&&A\otimes _{R}B,}
\end{equation*}
\begin{invisible}
  \begin{equation*}
R\overset{u_{A}}{\rightarrow }A\cong A\otimes _{R}R\overset{A\otimes
_{R}u_{B}}{\rightarrow }A\otimes _{R}B,
\end{equation*}%
\end{invisible}
i.e. the unit of $A\otimes _{R}B$, is an epimorphism. By Theorem \ref%
{thm:alg}, $\left( A\otimes _{R}B\right) /R$ is h-separable.
\begin{invisible}
Another proof: $u_{B}$ epi$\Rightarrow A\otimes _{R}u_{B}$ epi.

$f,g:A\otimes _{R}B\rightarrow E$ ring map such that $f\circ \left( A\otimes
_{R}u_{B}\right) =g\circ \left( A\otimes _{R}u_{B}\right) .$

$f\left( a\otimes _{R}b\right) =f\left( \left( a\otimes _{R}1_{B}\right)
\left( 1_{A}\otimes _{R}b\right) \right) =f\left( a\otimes _{R}1_{B}\right)
f\left( 1_{A}\otimes _{R}b\right) =f\left( a\otimes _{R}u_{B}\left(
1_{R}\right) \right) f\left( 1_{A}\otimes _{R}b\right) $

$g\left( a\otimes _{R}b\right) =g\left( a\otimes _{R}u_{B}\left(
1_{R}\right) \right) g\left( 1_{A}\otimes _{R}b\right) $

$f\left( 1_{A}\otimes _{R}b\right) =?=g\left( 1_{A}\otimes _{R}b\right) $

$f\circ \left( A\otimes _{R}u_{B}\right) =g\circ \left( A\otimes
_{R}u_{B}\right) \Rightarrow f\left( 1_{A}\otimes _{R}u_{B}\left( r\right)
\right) =g\left( 1_{A}\otimes _{R}u_{B}\left( r\right) \right) $

$\Rightarrow f\left( 1_{A}\otimes _{R}-\right) u_{B}\left( r\right) =g\left(
1_{A}\otimes _{R}-\right) u_{B}\left( r\right) \Rightarrow f\left(
1_{A}\otimes _{R}-\right) \circ u_{B}=g\left( 1_{A}\otimes _{R}-\right)
\circ u_{B}$

$\Rightarrow f\left( 1_{A}\otimes _{R}-\right) =g\left( 1_{A}\otimes
_{R}-\right) .$
\end{invisible}
\end{proof}

\begin{remark}
Let $R$ be a ring, $G$ be a group and consider  $RG$, the group ring. S.
Caenepeel posed the following problem: to characterize whether $RG/R\ $is
h-separable. In general we do not have an answer to this question. However,
if $R$ is commutative, we can consider a maximal ideal $M$ of $R$ and take $%
\Bbbk :=R/M.$ By Lemma \ref{Lem:otalg}, we deduce that $\left( \Bbbk \otimes
_{R}RG\right) /\Bbbk $ is h-separable i.e. $\Bbbk G/\Bbbk $ is h-separable.
By Proposition \ref{Pro:hsepoverfield}, we conclude that $\Bbbk G=\Bbbk $
and hence $\left\vert G\right\vert =1.$
\end{remark}

\section{Example on monoidal categories}\label{sec:4}

In the present section $\mathcal{M}$ denotes a preadditive braided monoidal
category such that

\begin{itemize}
\item $\mathcal{M}$ has equalizers and denumerable coproducts;

\item the tensor products are additive and preserve equalizers and
denumerable coproducts.
\end{itemize}

In view of the assumptions above, we can apply \cite[Theorem 4.6]%
{AM-BraidedOb} to obtain an adjunction $\left(
\mathbf{T},\mathbf{P}\right) $ as in the following diagram
\begin{invisible}
  \begin{equation*}
\begin{array}{ccc}
\mathrm{Bialg}\left( \mathcal{M}\right) & \overset{\mho }{\longrightarrow }
& \mathrm{Alg}\left( \mathcal{M}\right) \\
\mathbf{T}\uparrow \downarrow \mathbf{P} &  & T\uparrow \downarrow \Omega \\
\mathcal{M} & \overset{\mathrm{Id}}{\longrightarrow } & \mathcal{M}%
\end{array}
\end{equation*}
\end{invisible}
\begin{equation*}
\xymatrixcolsep{0.7cm}\xymatrixrowsep{0.7cm}
\xymatrix{\Bialg(\M)\ar[rr]^{\mho}\ar@<.5ex>[d]^{\mathbf{P}}&&\Alg(\M)\ar@<.5ex>[d]^{\Omega}\\
\M\ar[rr]^{\id}\ar@<.5ex>[u]^{\mathbf{T}}&&\M\ar@<.5ex>[u]^{T} }
\end{equation*}
Here $\Alg(\M)$ denotes the category of algebras (or monoids) in $\M$, $\Bialg(\M)$ is the category of bialgebras (or bimonoids) in $\M$, the functors $\mho$ and $\Omega$ are the obvious forgetful functors and, by construction of $\mathbf{T}$, we have $\mho \circ \mathbf{T}=T.$

It is noteworthy that, since $\Omega $ has a left adjoint $T$, then $\Omega $
is strictly monadic (the comparison functor is a category isomorphism), see
\cite[Theorem A.6]{AM-MM}.

Let $V\in \mathcal{M}.$ By construction $\Omega TV=\oplus _{n\in \mathbb{N}%
}V^{\otimes n},$ see \cite[Remark 1.2]{AM-BraidedOb}. Denote by $\alpha
_{n}V:V^{\otimes n}\rightarrow \Omega TV$ the canonical inclusion. The unit
$\eta :\mathrm{Id}_{\mathcal{%
M}}\rightarrow \Omega T$ of the adjunction $\left( T,\Omega \right) $ is defined by $\eta V:=\alpha _{1}V$ while the counit $%
\epsilon :T\Omega \rightarrow \mathrm{Id}$ is uniquely defined by the
equality
\begin{equation}
\Omega \epsilon \left( A,m,u\right) \circ \alpha _{n}A=m^{n-1}\text{ for
every }n\in \mathbb{N}  \label{form:epsTOmega}
\end{equation}%
where $m^{n-1}:A^{\otimes n}\rightarrow A$ denotes the iterated
multiplication of an algebra $\left( A,m,u\right) $ defined by $%
m^{-1}=u,m^{0}=\mathrm{Id}_{A}$ and, for $n\geq 2,m^{n-1}=m\circ \left(
m^{n-2}\otimes A\right) .$

Denote by $\boldsymbol{\eta },\boldsymbol{\epsilon }$ the unit and counit of
the adjunction $\left( \mathbf{T},\mathbf{P}\right) $.

Consider the natural transformation $\xi :\mathbf{P}\rightarrow \Omega \mho $
defined by%
\begin{equation*}
\xymatrixcolsep{1cm}\xymatrixrowsep{0.7cm}
\xymatrix{\mathbf{P}\ar[r]^-{\eta P}&\Omega T\mathbf{P}=\Omega \mho
\mathbf{T}\mathbf{P}\ar[r]^-{\Omega \mho \boldsymbol{\epsilon }} &\Omega \mho .}
\end{equation*}
\begin{invisible}
\begin{equation*}
\mathbf{P}\overset{\eta P}{\rightarrow }\Omega T\mathbf{P}=\Omega \mho
\mathbf{T}\mathbf{P}\overset{\Omega \mho \boldsymbol{\epsilon }}{\rightarrow
}\Omega \mho .
\end{equation*}%
\end{invisible}
We have $\epsilon \mho \circ T\xi =\epsilon \mho \circ T\Omega \mho \boldsymbol{%
\epsilon }\circ T\eta P=\mho \boldsymbol{\epsilon }\circ \epsilon T\mathbf{P}%
\circ T\eta \mathbf{P}=\mho \boldsymbol{\epsilon }$ i.e.%
\begin{equation}
\epsilon \mho \circ T\xi =\mho \boldsymbol{\epsilon }.  \label{form:epstilde}
\end{equation}%
so that $\xi $ is exactly the natural transformation of \cite[Theorem 4.6]%
{AM-BraidedOb}, whose components are the canonical inclusions of the
subobject of primitives of a bialgebra $B$ in $\mathcal{M}$ into $\Omega
\mho B$ and hence they are regular monomorphisms.

Define the functor
\begin{equation*}
\left( -\right) ^{+}:\mathrm{Bialg}\left( \mathcal{M}\right) \rightarrow
\mathcal{M}
\end{equation*}%
that assigns to every bialgebra $A$ the kernel $\left( A^{+},\zeta
A:A^{+}\rightarrow \Omega \mho A\right) $ of $\varepsilon _{\Omega \mho A}$
(i.e. the equalizer of $\varepsilon _{\Omega \mho A}:\Omega \mho
A\rightarrow \mathbf{1}$ and the zero morphism) and to every morphism $f$
the induced morphism $f^{+}.$

Since $\zeta A$ is natural in $A$ we get a natural transformation $\zeta
:\left( -\right) ^{+}\rightarrow \Omega \mho $ which is by construction a
monomorphism on components.

\begin{lemma}
\label{lem:equifib}The natural transformation $\xi :\mathbf{P}\rightarrow
\Omega \mho $ factors through the natural transformation $\zeta :\left(
-\right) ^{+}\rightarrow \Omega \mho $ (i.e. there is $\widehat{\xi }:%
\mathbf{P}\rightarrow \left( -\right) ^{+}$ such that $\xi =\zeta \circ
\widehat{\xi }$) which is a monomorphism on components.
\end{lemma}

\begin{proof}
Given $A\in \mathrm{Bialg}\left( \mathcal{M}\right) $ we have that $\xi $
and $\zeta $ are defined by the following kernels.
\begin{equation*}
\xymatrixcolsep{2.5cm}\xymatrixrowsep{0.7cm}
\xymatrix{\mathbf{P}A\ar[r]^{\xi A}\ar[d]^{\widehat{\xi}}&\Omega \mho A \ar[d]^{\id} \ar[rrr]^{%
\left( u_{\Omega \mho A}\otimes \Omega \mho A\right) r_{\Omega \mho
A}^{-1}+\left( \Omega \mho A\otimes u_{\Omega \mho A}\right) \circ l_{\Omega
\mho A}^{-1}-\Delta _{\Omega \mho A}}&&&\Omega \mho
A\otimes \Omega \mho A \ar[d]_{m_{\mathbf{1}}\left( \varepsilon _{\Omega \mho A}\otimes \varepsilon _{\Omega
\mho A}\right)}\\
A^{+} \ar[r]^{\zeta A}&\Omega \mho A \ar[rrr]^{\varepsilon _{\Omega \mho A}}&&&\mathbf{1}}
\end{equation*}
\begin{invisible}
 \begin{equation*}
\begin{array}{ccccc}
\mathbf{P}A & \overset{\xi A}{\longrightarrow } & \Omega \mho A & \overset{%
\left( u_{\Omega \mho A}\otimes \Omega \mho A\right) r_{\Omega \mho
A}^{-1}+\left( \Omega \mho A\otimes u_{\Omega \mho A}\right) \circ l_{\Omega
\mho A}^{-1}-\Delta _{\Omega \mho A}}{\longrightarrow } & \Omega \mho
A\otimes \Omega \mho A \\
\downarrow \widehat{\xi }A &  & \downarrow \mathrm{Id} &  & \downarrow m_{%
\mathbf{1}}\left( \varepsilon _{\Omega \mho A}\otimes \varepsilon _{\Omega
\mho A}\right) \\
A^{+} & \overset{\zeta A}{\longrightarrow } & \Omega \mho A & \overset{%
\varepsilon _{\Omega \mho A}}{\longrightarrow } & \mathbf{1}%
\end{array}%
\end{equation*}%
\end{invisible}
Since the right square above commutes, there is a unique morphism $\widehat{%
\xi }A:\mathbf{P}A\rightarrow A^{+}$ such that $\zeta A\circ \widehat{\xi }%
A=\xi A.$ The naturality of $\zeta A$ and $\xi A$ in $A$ implies the one of $%
\widehat{\xi }A\ $so that $\zeta \circ \widehat{\xi }=\xi .$
\end{proof}

There is a unique morphism $\omega V:\Omega TV\rightarrow V$ such that%
\begin{equation}
\omega V\circ \alpha _{n}V=\delta _{n,1}\mathrm{Id}_{V}.
\label{form:gamalpha}
\end{equation}%
Given $f:V\rightarrow W$ a morphism in $\mathcal{M},$ we get for every $n\in
\mathbb{N},$%
\begin{equation*}
\omega W\circ \Omega Tf\circ \alpha _{n}V=\omega W\circ \alpha _{n}W\circ
f^{\otimes n}=\delta _{n,1}f^{\otimes n}=\delta _{n,1}f=f\circ \omega V\circ
\alpha _{n}V
\end{equation*}%
so that $\omega W\circ \Omega Tf=f\circ \omega V$ which means that $\omega
:=\left( \omega V\right) _{V\in \mathcal{M}}$ is a natural transformation $%
\omega :\Omega T\rightarrow \mathrm{Id}_{\mathcal{M}}.$

\begin{lemma}
\label{lem:omega}The natural transformation $\omega $ fulfills $\omega \circ
\eta =\mathrm{Id}$ and%
\begin{equation}
\omega \omega \circ \Omega T\zeta \mathbf{T}=\omega \circ \Omega \epsilon
T\circ \Omega T\zeta \mathbf{T}.  \label{form:omega}
\end{equation}
\end{lemma}

\begin{proof}
We have%
\begin{equation*}
\omega V\circ \eta V=\omega V\circ \alpha _{1}V\overset{(\ref{form:gamalpha})%
}{=}\mathrm{Id}_{V}
\end{equation*}%
and hence $\omega \circ \eta =\mathrm{Id}.$ Let us check (\ref{form:omega}).
For every $V\in \mathcal{M}$ we compute%
\begin{eqnarray*}
&&\omega \omega V\circ \Omega T\zeta \mathbf{T}V\circ \alpha _{n}\left(
-\right) ^{+}\mathbf{T}V \\
&=&\omega V\circ \Omega T\omega V\circ \Omega T\zeta \mathbf{T}V\circ \alpha
_{n}\left( -\right) ^{+}\mathbf{T}V=\omega V\circ \alpha _{n}V\circ \left(
\omega V\right) ^{\otimes n}\circ \left( \zeta \mathbf{T}V\right) ^{\otimes
n} \\
&=&\delta _{n,1}\left( \omega V\right) ^{\otimes n}\circ \left( \zeta
\mathbf{T}V\right) ^{\otimes n}=\delta _{n,1}\omega V\circ \zeta \mathbf{T}V.
\end{eqnarray*}%
On the other hand%
\begin{eqnarray*}
\omega V\circ \Omega \epsilon TV\circ \Omega T\zeta \mathbf{T}V\circ \alpha
_{n}\left( -\right) ^{+}\mathbf{T}V &=&\omega V\circ \Omega \epsilon TV\circ
\alpha _{n}\Omega \mho \mathbf{T}V\circ \left( \zeta \mathbf{T}V\right)
^{\otimes n} \\
&=&\omega V\circ \Omega \epsilon TV\circ \alpha _{n}\Omega TV\circ \left(
\zeta \mathbf{T}V\right) ^{\otimes n} \\
&\overset{(\ref{form:epsTOmega})}{=}&\omega V\circ m_{\Omega TV}^{n-1}\circ
\left( \zeta \mathbf{T}V\right) ^{\otimes n}.
\end{eqnarray*}%
Hence we have to check that%
\begin{equation*}
\delta _{n,1}\omega V\circ \zeta \mathbf{T}V=\omega V\circ m_{\Omega
TV}^{n-1}\circ \left( \zeta \mathbf{T}V\right) ^{\otimes n}.
\end{equation*}%
For $n=0$ we have%
\begin{equation*}
\delta _{0,1}\omega V\circ \zeta \mathbf{T}V=0=\omega V\circ \alpha
_{0}V=\omega V\circ u_{\Omega TV}=\omega V\circ m_{\Omega TV}^{-1}\circ
\left( \zeta \mathbf{T}V\right) ^{\otimes 0}.
\end{equation*}%
For $n=1$ we have%
\begin{equation*}
\delta _{1,1}\omega V\circ \zeta \mathbf{T}V=\omega V\circ \zeta \mathbf{T}%
V=\omega V\circ m_{\Omega TV}^{0}\circ \left( \zeta \mathbf{T}V\right)
^{\otimes 1}.
\end{equation*}%
For $n\geq 2$ we have $\delta _{n,1}\omega V\circ \zeta \mathbf{T}V=0$. In
order to prove that also $\omega \circ m_{\Omega TV}^{n-1}\circ \left( \zeta
\mathbf{T}V\right) ^{\otimes n}=0$ we need first to give a different
expression for $\omega V\circ m_{\Omega TV}.$ To this aim we compute (we use
the identifications $V\otimes \mathbf{1}\cong V\cong \mathbf{1}\otimes V$)%
\begin{eqnarray*}
&&\omega V\circ m_{\Omega TV}\circ \left( \alpha _{m}V\otimes \alpha
_{n}V\right) \\
&=&\omega V\circ \alpha _{m+n}V=\delta _{m+n,1}\mathrm{Id}_{V} \\
&=&\delta _{m,1}\delta _{n,0}\mathrm{Id}_{V\otimes \mathbf{1}}+\delta
_{m,0}\delta _{n,1}\mathrm{Id}_{\mathbf{1\otimes }V} \\
&=&r_{V}\circ \left( \delta _{m,1}\mathrm{Id}_{V}\otimes \delta _{n,0}%
\mathrm{Id}_{\mathbf{1}}\right) +l_{V}\circ \left( \delta _{m,0}\mathrm{Id}_{%
\mathbf{1}}\otimes \delta _{n,1}\mathrm{Id}_{V}\right) \\
&=&r_{V}\circ \left( \omega V\otimes \varepsilon _{\Omega TV}\right) \circ
\left( \alpha _{m}V\otimes \alpha _{n}V\right) +l_{V}\circ \left(
\varepsilon _{\Omega TV}\otimes \omega V\right) \circ \left( \alpha
_{m}V\otimes \alpha _{n}V\right) \\
&=&\left( r_{V}\circ \left( \omega V\otimes \varepsilon _{\Omega TV}\right)
+l_{V}\circ \left( \varepsilon _{\Omega TV}\otimes \omega V\right) \right)
\circ \left( \alpha _{m}V\otimes \alpha _{n}V\right) .
\end{eqnarray*}%
Since the tensor products preserve denumerable coproducts, the equalities
above yield the identity%
\begin{equation*}
\omega V\circ m_{\Omega TV}=r_{V}\circ \left( \omega V\otimes \varepsilon
_{\Omega TV}\right) +l_{V}\circ \left( \varepsilon _{\Omega TV}\otimes
\omega V\right) .
\end{equation*}%
Using it, we obtain
\begin{eqnarray*}
&&\omega V\circ m_{\Omega TV}^{n-1}\circ \left( \zeta \mathbf{T}V\right)
^{\otimes n} \\
&=&\omega V\circ m_{\Omega TV}\circ \left( m_{\Omega TV}^{n-2}\otimes \Omega
TV\right) \circ \left( \zeta \mathbf{T}V\right) ^{\otimes n} \\
&=&\left( r_{V}\circ \left( \omega V\otimes \varepsilon _{\Omega TV}\right)
+l_{V}\circ \left( \varepsilon _{\Omega TV}\otimes \omega V\right) \right)
\circ \left( m_{\Omega TV}^{n-2}\otimes \Omega TV\right) \circ \left( \zeta
\mathbf{T}V\right) ^{\otimes n} \\
&=&r_{V}\circ \left( \omega V\circ m_{\Omega TV}^{n-2}\otimes \varepsilon
_{\Omega TV}\right) \circ \left( \zeta \mathbf{T}V\right) ^{\otimes
n}+l_{V}\circ \left( \varepsilon _{\Omega TV}\circ m_{\Omega
TV}^{n-2}\otimes \omega V\right) \circ \left( \zeta \mathbf{T}V\right)
^{\otimes n} \\
&=&r_{V}\circ \left( \omega V\circ m_{\Omega TV}^{n-2}\circ \left( \zeta
\mathbf{T}V\right) ^{\otimes n-1}\otimes \varepsilon _{\Omega TV}\circ \zeta
\mathbf{T}V\right) +l_{V}\circ \left( \left( \varepsilon _{\Omega TV}\circ
\zeta \mathbf{T}V\right) ^{\otimes n-1}\otimes \omega V\circ \zeta \mathbf{T}%
V\right).
\end{eqnarray*}%
The last two summands are zero as $\varepsilon _{\Omega TV}\circ \zeta \mathbf{T}V=\varepsilon _{\Omega\mho \mathbf{T}V}\circ \zeta \mathbf{T}V=0$ by definition of $\zeta$.
\end{proof}

\begin{remark}
\label{rem:Tnotheavy}As observed the comparison functor $K:\mathrm{Alg}%
\left( \mathcal{M}\right) \rightarrow \mathcal{M}_{\Omega T}$ is an
isomorphism of categories. By Corollary \ref{coro:heavsepU}, $T$ is
h-separable if and only if $\Omega :\mathrm{Alg}\left( \mathcal{M}\right)
\rightarrow \mathcal{M}$ is a split epimorphism. Let us prove, by contradiction, that this is not the case. Assume that there is a functor $\Gamma :\mathcal{M}\rightarrow
\mathrm{Alg}\left( \mathcal{M}\right) $ such that $\Omega \Gamma =\mathrm{Id}%
.$ Let $V\in \mathcal{M}$. Then $\Gamma V=\left( V,mV,uV\right) $ for some
morphisms $mV:V\otimes V\rightarrow V$ and $uV:\mathbf{1}\rightarrow V.$ Let $%
f:V\rightarrow V$ be the zero morphism. Then $\Gamma f:\Gamma V\rightarrow \Gamma
V$ is an algebra morphism and $\Omega \Gamma f=f$. Thus $f$ is unitary i.e. $%
uV=f\circ uV=0\circ uV=0$. Hence $\mathrm{Id}_{V}=mV\circ \left( V\otimes
uV\right) \circ r_{V}^{-1}=0.$ As a consequence any morphism $h:V\rightarrow
W$ would be zero as $h=h\circ \mathrm{Id}_{V}$ for every $V,W\in \mathcal{M}$%
. Hence $\mathrm{Hom}_{\mathcal{M}}\left( V,W\right) =\left\{ 0\right\} .$
This happens only if all objects are isomorphic to the unit object $\mathbf{1%
}$, i.e. if the skeleton of $\mathcal{M}$ is the trivial monoidal category $%
\left( \mathcal{T},\otimes ,\mathbf{1}\right) $, where $\mathrm{Ob}\left(
\mathcal{T}\right) =\left\{ \mathbf{1}\right\} ,$ $\mathrm{Hom}_{\mathcal{T}%
}\left( \mathbf{1,1}\right) =\left\{ \mathrm{Id}_{\mathbf{1}}\right\} $ and
the tensor product is given by $\mathbf{1\otimes 1=1}$ and $\mathrm{Id}_{%
\mathbf{1}}\otimes \mathrm{Id}_{\mathbf{1}}=\mathrm{Id}_{\mathbf{1}}.$ This
is evidently a restrictive condition on $\mathcal{M}$. Thus, in general $T:%
\mathcal{M}\rightarrow \mathrm{Alg}\left( \mathcal{M}\right) $ is not
heavily separable. On the other hand the equality $\omega \circ \eta =%
\mathrm{Id}$ obtained in Lemma \ref{lem:omega} means that the functor $T:%
\mathcal{M}\rightarrow \mathrm{Alg}\left( \mathcal{M}\right) $ is separable.

As a particular case, we get that the functor $T:\vec\rightarrow \mathrm{Alg}_{\Bbbk }$ is separable but not h-separable.
\end{remark}

\begin{theorem}
Set $\boldsymbol{\gamma} :=\omega \circ \xi \mathbf{T:PT}\rightarrow \mathrm{Id}_{%
\mathcal{M}}.$ Then $\boldsymbol{\gamma} \circ \boldsymbol{\eta }=\mathrm{Id}$ and $%
\boldsymbol{\gamma} \boldsymbol{\gamma} =\boldsymbol{\gamma} \circ \mathbf{\mathbf{P\boldsymbol{\epsilon }T.}}$
Hence the functor $\mathbf{T}:\mathcal{M}\rightarrow \mathrm{Bialg}\left(
\mathcal{M}\right) $ is h-separable.
\end{theorem}

\begin{proof}
We compute%
\begin{eqnarray*}
\boldsymbol{\gamma} \circ \boldsymbol{\eta } &=&\omega \circ \xi \mathbf{T}\circ
\boldsymbol{\eta }\overset{\text{def. }\xi }{=}\omega \circ \Omega \mho
\boldsymbol{\epsilon }\mathbf{T}\circ \eta \mathbf{PT}\circ \boldsymbol{\eta
} \\
&=&\omega \circ \Omega \mho \boldsymbol{\epsilon }\mathbf{T}\circ \Omega T%
\boldsymbol{\eta }\circ \eta =\omega \circ \Omega \mho \boldsymbol{\epsilon }%
\mathbf{T}\circ \Omega \mho \mathbf{T}\boldsymbol{\eta }\circ \eta =\omega
\circ \eta =\mathrm{Id}.
\end{eqnarray*}%
Moreover%
\begin{eqnarray*}
\Omega \epsilon \mho\circ \xi \mathbf{\mathbf{T}}\xi 
&=&\Omega \epsilon \mho\circ \Omega \mho \mathbf{\mathbf{T}}\xi \circ \xi \mathbf{\mathbf{TP}} \\
&\overset{\text{def.}\xi }{=}&\Omega \epsilon \mho\circ \Omega T\Omega \mho
\boldsymbol{\epsilon }\circ \Omega T\eta P\circ \Omega \mho \boldsymbol{\epsilon }\mathbf{\mathbf{TP}}\circ \eta P\mathbf{\mathbf{TP}} \\
&=&\Omega \mho \boldsymbol{\epsilon }\circ \Omega
\epsilon \mho \mathbf{\mathbf{T\mathbf{P}}}\circ \Omega T\eta P\circ \Omega \mho \boldsymbol{\epsilon }\mathbf{\mathbf{TP}}%
\circ \eta P\mathbf{\mathbf{TP}}  \\
&=&\Omega \mho \boldsymbol{\epsilon }\circ \Omega \mho
\boldsymbol{\epsilon }\mathbf{\mathbf{TP}}\circ \eta P\mathbf{\mathbf{TP}}
\\
&=&\Omega \mho \boldsymbol{\epsilon }\circ \Omega \mho
\mathbf{\mathbf{TP\boldsymbol{\epsilon }}}\circ \eta P\mathbf{\mathbf{TP}}
\\
&=&\Omega \mho \boldsymbol{\epsilon }\circ \eta P\circ \mathbf{\mathbf{P\boldsymbol{\epsilon }}}\overset{\text{%
def.}\xi }{=}\xi \circ \mathbf{\mathbf{P\boldsymbol{%
\epsilon }}}
\end{eqnarray*}%
so that%
\begin{gather*}
\boldsymbol{\gamma} \boldsymbol{\gamma} =\omega \omega \circ \xi \mathbf{T}\xi \mathbf{T}=\omega
\omega \circ \Omega \mho \mathbf{T}\zeta \mathbf{T\circ }\xi \mathbf{T}%
\widehat{\xi }\mathbf{T}=\omega \omega \circ \Omega T\zeta \mathbf{T\circ }%
\xi \mathbf{T}\widehat{\xi }\mathbf{T} \\
\overset{(\ref{form:omega})}{=}\omega \circ \Omega \epsilon T\circ \Omega
T\zeta \mathbf{T\circ }\xi \mathbf{T}\widehat{\xi }\mathbf{T}=\omega \circ
\Omega \epsilon \mho\mathbf{T}\circ \xi \mathbf{T}\xi \mathbf{T}=\omega \circ \xi \mathbf{%
\mathbf{T}}\circ \mathbf{\mathbf{P\boldsymbol{\epsilon }T}}=\boldsymbol{\gamma} \circ
\mathbf{\mathbf{P\boldsymbol{\epsilon }T}}
\end{gather*}
\end{proof}

\end{document}